\documentclass[12pt]{amsart}
\usepackage{amsfonts, amsmath, amsthm}

\usepackage{graphicx}
\usepackage{epstopdf}
\usepackage{psfrag}
\newtheorem{pro}{Proposition}[section]
\newtheorem{thm}[pro]{Theorem}
\newtheorem{lem}[pro]{Lemma}
\newtheorem{clm}[pro]{Claim}

\newtheorem{cor}[pro]{Corollary}
\newtheorem{quest}[pro]{Question}

\theoremstyle{definition}
\newtheorem{dfn}[pro]{Definition}

\theoremstyle{remark}

\newcommand{\VV}{\mathcal V}
\newcommand{\WW}{\mathcal W}
\newcommand{\CC}{\mathcal C}

\newcommand{\gen}{\mbox{\rm genus}}

\newcommand{\bdy}{\partial}

\newcommand{\thick}[1]{{\rm Thick}(#1)}
\newcommand{\thin}[1]{{\rm Thin}(#1)}

\newcommand{\amlg}[1]{\mathcal A(#1)}

\newcommand{\CV}{\mathcal V}
\newcommand{\CW}{\mathcal W}

\title{Stabilizing and destabilizing Heegaard splittings of sufficiently complicated 3-manifolds}
\date{\today}
\address{Pitzer College}
\email{bachman@pitzer.edu}
\author{David Bachman}
\begin{document}

\begin{abstract}
Let $M_1$ and $M_2$ be compact, orientable 3-manifolds with incompressible boundary, and  $M$ the manifold obtained by gluing with a homeomorphism $\phi:\bdy M_1 \to \bdy M_2$. We analyze the relationship between the sets of low genus Heegaard splittings of $M_1$, $M_2$, and $M$, assuming the map $\phi$ is ``sufficiently complicated." This analysis yields counter-examples to the Stabilization Conjecture, a resolution of the higher genus analogue of a conjecture of Gordon, and a result about the uniqueness of expressions of Heegaard splittings as amalgamations.
\end{abstract}

\maketitle

\markright{HEEGAARD SPLITTINGS OF SUFFICIENTLY COMPLICATED 3-MANIFOLDS}

\section{Introduction}

Suppose $M_1$ and $M_2$ are 3-manifolds with incompressible, homeomorphic boundary components, $F_i \subset \bdy M_i$. Let $M$ be the 3-manifold obtained from these manifolds by gluing via the map $\phi: F_1 \to F_2$.  A fundamental question is to determine the extent to which the set of Heegaard splittings of $M$ is determined by the sets of Heegaard splittings of $M_1$ and $M_2$. In this paper we analyze the relationships between  these sets, assuming the gluing map $\phi$ is ``sufficiently complicated," with respect to some measure of complexity, and the splittings under consideration are ``low genus," in relation to this complexity. 

The main technical idea from which our results follow is that when $M_1$ is glued to $M_2$ by a complicated homeomorphism, the surface $F$ at their interface becomes a ``barrier" to all low genus {\it topologically minimal} surfaces. Such surfaces are the topological analogue of geometrically minimal surfaces, and the intuition behind why  $F$ should be a barrier to such surfaces comes from geometry. As the gluing map between $M_1$ and $M_2$ becomes more and more complicated, we see a longer and longer region of $M_1 \cup _\phi M_2$ which is homeomorphic to $F \times I$. Any minimal surface which passes through such a region must have large area, and thus by the Gauss-Bonnet theorem have large genus. Lackenby used precisely this idea to study the behavior of Heegaard genus under complicated gluings \cite{lackenby:04}. 

Examples of topologically minimal surfaces include {\it incompressible}, {\it strongly irreducible}, and {\it criticial} surfaces. Knowing that the surface $F$ is a barrier to these three types of surfaces  implies that all stabilizations, destabilizations, and isotopies of Heegaard splittings of $M_1 \cup _\phi M_2$ happen away from $F$, i.e.~in either $M_1$ or $M_2$. In other words, the relationships between the low genus Heegaard splittings of $M_1 \cup_\phi M_2$ are completely determined by the relationships between Heegaard splittings of $M_1$ and $M_2$. From this, we obtain counter-examples to the Stabilization Conjecture, a resolution of the higher genus analogue of a conjecture of Gordon, and a result about the uniqueness of expressions of Heegaard splittings as amalgamations. We describe each of these presently.

\subsection{The Stabilization Conjecture.} Given a Heegaard surface $H$ in a 3-manifold, $M$, one can {\it stabilize} to obtain a splitting of higher genus by taking the connected sum of $H$ with the genus one splitting of $S^3$. Suppose $H_1$ and $H_2$ are Heegaard surfaces in $M$, where $\rm{genus}(H_1) \ge \rm{genus}(H_2)$. It is a classical result of Reidemeister \cite{reidemeister} and Singer \cite{singer} from 1933 that as long as $H_1$ and $H_2$ induce the same partition of the components of $\partial M$, stabilizing $H_1$ some number of times produces a stabilization of $H_2$. Just one stabilization was proved to be always sufficient in large classes of 3-manifolds, including Seifert fibered spaces \cite{schultens:96}, most genus two 3-manifolds \cite{rs:99}, \cite{bs:11},  and most graph manifolds \cite{Derby-Talbot2006} (see also \cite{sedgwick:97}). The lack of examples to the contrary led to ``The Stabilization Conjecture": Any pair of Heegaard splittings requires at most one stabilization to become equivalent. (See Conjecture 7.4 in \cite{ScharlemannSurvey}.)

In this paper we produce several families of counter-examples to the Stabilization Conjecture (see Section \ref{s:CounterExamples}). This work was announced in December of 2007 at a Workshop on {\it Triangulations, Heegaard Splittings, and Hyperbolic Geometry}, at the American Institute of Mathematics. At the same conference another family of counter-examples to the Stabilization Conjecture was announced by Hass, Thompson, and Thurston \cite{htt:09}. Their construction uses mainly geometric techniques. Several months later Johnson announced still more counter-examples \cite{johnson2}. The key to the constructions of the counter-examples given in \cite{htt:09} and \cite{johnson2}  is to use Heegaard splittings formed by gluing together two handlebodies by a very complicated homeomorphism. The construction here is quite different, and our techniques lead to the resolution of several other questions about Heegaard splittings, as described below.

\subsection{The higher genus analogue of Gordon's Conjecture.} In Problem 3.91 of \cite{kirby:97}, Cameron Gordon conjectured that the connected sum of  unstabilized Heegaard splittings is unstabilized. This was proved by the author in \cite{gordon}, and by Scharlemann and Qiu in \cite{sq:09}. 

Given Heegaard surfaces $H_i \subset M_i$, Schultens gave a construction of a Heegaard surface in $M=M_1 \cup _{\phi} M_2$, called their {\it amalgamation} \cite{schultens:93}. Using this terminology, we can phrase the higher genus analogue of Gordon's conjecture:

\begin{quest}
\label{c:IsotopyConjecture}
Let $M_1$ and $M_2$ denote compact, orientable, irreducible 3-manifolds with homeomorphic, incompressible boundary. Let $M$ be the 3-manifold obtained from $M_1$ and $M_2$ by gluing a component of each of their boundaries by some homeomorphism. Let $H_i$ be an unstabilized Heegaard surface in $M_i$. Is the amalgamation of $H_1$ and $H_2$ in $M$ unstabilized? 
\end{quest}

As stated, Schultens and Weidmann have shown the answer to this question is no \cite{SchultensWeidmann}. In light of their examples we refine the question by adding the hypothesis that the gluing map between $M_1$ and $M_2$ is  ``sufficiently complicated," in some suitable sense. We will postpone a precise definition of this term to Section \ref{s:DistanceBoundSection}. However, throughout this paper it will be used in such a way so that if $\psi:F^1 \to F^2$ is a fixed homeomorphism, then for each pseudo-Anosov map $\phi:F^2 \to F^2$, there exists an $N$ so that for each $n \ge N$, $\psi^{-1}\phi^n \psi$ is sufficiently complicated. Unfortunately, the assumption that the gluing map of Question \ref{c:IsotopyConjecture} is sufficiently complicated is still not a strong enough hypothesis to insure the answer is yes, as the following construction shows.

If $\bdy M \ne \emptyset$, then one can {\it boundary-stabilize} a Heegaard surface in $M$ by tubing a copy of a component of $\bdy M$ to it \cite{moriah:02}.  Let $M_1$ be a manifold that has a boundary component $F$, and an unstabilized Heegaard surface $H_1$ that has been obtained by boundary-stabilizing some other Heegaard surface along $F$. (See \cite{sedgwick:01} or \cite{ms:04} for such examples.) Let $M_2$ be a manifold with a boundary component homeomorphic to $F$, and a {\it $\gamma$-primitive} Heegaard surface (see \cite{moriah:02}). Such a Heegaard surface is unstabilized, but has the property that boundary-stabilizing it along $F$ produces a stabilized Heegaard surface. Then no matter how we glue $M_1$ to $M_2$ along $F$, the amalgamation of $H_1$ and $H_2$ will be stabilized. 

Given this example, and those of Schultens and Weidmann, we deduce the following: In order for the answer to Question \ref{c:IsotopyConjecture} to be yes, we would at least have to know that $H_1$ and $H_2$ are not stabilized, not boundary-stabilized, and that the gluing map is sufficiently complicated. We prove here that these hypotheses are enough to obtain the desired result:

\medskip

\noindent {\bf Theorem \ref{t:HigherGenusGordon}. }{\it
Let $M_1$ and $M_2$ be compact, orientable, irreducible 3-manifolds with incompressible boundary, neither of which is an $I$-bundle. Let $M$ denote the manifold obtained by gluing some component $F$ of $\bdy M_1$ to some component of $\bdy M_2$ by some homeomorphism $\phi$. Let $H_i$ be an unstabilized, boundary-unstabilized Heegaard surface in $M_i$. If $\phi$ is sufficiently complicated then the amalgamation of $H_1$ and $H_2$ in $M$ is unstabilized.}

\medskip

This result allows us to construct the first example of a non-minimal genus Heegaard surface which has Hempel distance \cite{hempel:01} exactly one.  The first examples of minimal genus, distance one Heegaard surfaces were found by Lustig and Moriah in 1999  \cite{moriah:99}. Since then the existence of non-minimal genus examples was expected, but a construction remained elusive. This is why Moriah has called the search for such examples the ``nemesis of Heegaard splittings" \cite{moriah}.  In Corollary \ref{c:Moriah} we produce manifolds that have an arbitrarily large number of such splittings.

\subsection{The uniqueness of amalgamations.} The conclusion of Theorem \ref{t:HigherGenusGordon} asserts that each pair of low genus, unstabilized, boundary-unstabilized surfaces in $M_1$ and $M_2$ determines an unstabilized surface in $M_1 \cup _\phi M_2$. We now discuss the converse of this statement. Lackenby \cite{lackenby:04}, Souto \cite{souto}, and Li \cite{li:08} have independently shown that when $\phi$ is sufficiently complicated, then any low genus Heegaard surface $H$ in $M_1 \cup _\phi M_2$ is an amalgamation of Heegaard surfaces $H_i$ in $M_i$. In Theorem \ref{t:AmalgamationExists} we prove a slight refinement of this result:

\medskip

\noindent {\bf Theorem \ref{t:AmalgamationExists}. }{\it
Let $M_1$ and $M_2$ be compact, orientable, irreducible 3-manifolds with incompressible boundary, neither of which is an $I$-bundle. Let $M$ denote the manifold obtained by gluing some component $F$ of $\bdy M_1$ to some component of $\bdy M_2$ by some homeomorphism $\phi$. If $\phi$ is sufficiently complicated then any low genus, unstabilized Heegaard surface in $M$ is an amalgamation of unstabilized, boundary-unstabilized Heegaard surfaces in $M_1$ and $M_2$, and possibly a Type II splitting of $F \times I$.}

\medskip

Here a Type II splitting of $F \times I$ consists of two copies of $F$ connected by an unknotted tube (see \cite{st:93}). Suppose, as in the theorem above, that $F$ is a boundary component of $M_1$, and $H_1$ is a Heegaard surface in $M_1$. If we glue $F \times I$ to $\bdy M_1$, and amalgamate $H_1$ with a Type II splitting of $F \times I$, then the result is the same as if we had just boundary-stabilized $H_1$. 

Ideally, we would like to say that the Heegaard surfaces in $M_i$ given by Theorem \ref{t:AmalgamationExists} are uniquely determined by the Heegaard surface in $M$ from which they come. However, no matter how complicated $\phi$ is this may not be the case, as the following construction shows.

Let $M_1$ be a 3-manifold with boundary homeomorphic to a surface $F$, that has inequivalent unstabilized, boundary-unstabilized Heegaard surfaces $H_1$ and $G_1$ that become equivalent after a boundary-stabilization. (For example, $M_1$ may be a Seifert fibered space with a single boundary component. {\it Vertical} splittings $H_1$ and $G_1$ would then be equivalent after a boundary stabilization, by \cite{schultens:96}.) Let $M_2$ be any 3-manifold with boundary homeomorphic to $F$, and let $H_2$ be an unstabilized, boundary-unstabilized Heegaard surface in $M_2$. Glue $M_1$ to $M_2$ by any map $\phi$ to create the manifold $M$. Let $H$ be the amalgamation of $H_1$, $H_2$, and a Type II splitting of $F \times I$. Then $H$ is also the amalgamation of $G_1$, $H_2$ and a Type II splitting of $F \times I$. So the expression of $H$ as an amalagamation as described by the conclusion of Theorem \ref{t:AmalgamationExists} is not unique. 

This construction shows that Type II splittings of $F \times I$ are obstructions to the uniqueness of the decomposition given by Theorem \ref{t:AmalgamationExists}. In our final theorem, we show that this is the only obstruction:

\medskip

\noindent {\bf Theorem \ref{t:HighGenusGordonIsotopy}. }{\it
Let $M_1$ and $M_2$ be compact, orientable, irreducible 3-manifolds with incompressible boundary, neither of which is an $I$-bundle. Let $M$ denote the manifold obtained by gluing some component $F$ of $\bdy M_1$ to some component of $\bdy M_2$ by some homeomorphism $\phi$. Suppose $\phi$ is sufficiently complicated, and some low genus Heegaard surface  $H$ in $M$ can be expressed as an amalgamation of unstabilized, boundary-unstabilized Heegaard surfaces in $M_1$ and $M_2$. Then this expression is unique.}

\medskip

This paper is organized as follows. In Section \ref{s:FirstDefSection} we review the definitions of {\it incompressible}, {\it strongly irreducible}, and {\it critical} surfaces. In Section \ref{s:TopMinSfcs} we review a generalization of these three types of surfaces, the so-called {\it topologically minimal surfaces} of \cite{TopIndexI}. In Section \ref{s:IndexRelBdy} we further generalize this to consider surfaces that are topologically minimal {\it with respect to the boundary} of a manifold, and show how such surfaces can be obtained from topologically minimal ones. This allows us, in Section \ref{s:DistanceBoundSection}, to prove that when two 3-manifolds are glued together by a complicated map, it creates a ``barrier" to all low genus, topologically minimal surfaces. In particular, by gluing component manifolds together we create surfaces which act as obstructions to certain low genus incompressible, strongly irreducible, and critical surfaces. Such surfaces are thus dubbed {\it barrier surfaces}. In Section \ref{s:GHS} through \ref{s:LastDefSection} we mostly review the definitions and results given in \cite{gordon}. These include  {\it Generalized Heegaard splittings} (GHSs) and {\it Sequences of GHSs} (SOGs). We also define the {\it genus} of such objects, and establish several results about low genus GHSs and SOGs of manifolds that contain barrier surfaces. Section \ref{s:CounterExamples} contains our constructions of counter-examples to the Stabilization Conjecture. Sections \ref{s:Amalgamations} through \ref{s:Isotopy} contain the proofs of Theorems \ref{t:HigherGenusGordon}, \ref{t:AmalgamationExists}, and \ref{t:HighGenusGordonIsotopy}, mentioned above. 

The author thanks Tao Li for helpful conversations regarding his paper, \cite{li:08}, on which Section \ref{s:DistanceBoundSection} is based. Comments from Saul Schleimer, Ryan Derby-Talbot, and Sangyop Lee were also very helpful. Finally, the author thanks the referee for doing a careful reading and providing many helpful comments.

\section{Incompressible, Strongly Irreducible, and Critical surfaces}
\label{s:FirstDefSection}

In this section we recall the definitions of various classes of topologically interesting surfaces. The first are the {\it incompressible} surfaces of Haken \cite{haken:68}, which have played a central role in 3-manifold topology. The second class are the {\it strongly irreducible} surfaces of Casson and Gordon \cite{cg:87}. These surfaces have become important in answering a wide variety of questions relating to the Heegaard genus of 3-manifolds. The third class are the {\it critical surfaces} of \cite{crit} and \cite{gordon}. 

In \cite{TopIndexI} we show that all three of these classes are special cases of {\it topologically minimal} surfaces. Such surfaces are the topological analogue of geometrically minimal surfaces. We will say more about this in Section \ref{s:TopMinSfcs}. 

For the following definitions, $M$ will denote a compact, orientable 3-manifold.

\begin{dfn}
\label{d:essential}
Let $F$ be a properly embedded surface in $M$. Let $\gamma$ be a loop in $F$. $\gamma$ is {\it essential} on $F$ if it is a loop that does not bound a disk in $F$.  A {\it compression} for $F$ is a disk, $D$, such that $D \cap F=\bdy D$ is essential on $F$.
\end{dfn}

\begin{dfn}
Let $F$ be a properly embedded surface in $M$. The surface $F$ is {\it compressible} if there is a compression for it. Otherwise it is {\it incompressible}.
\end{dfn}

\begin{dfn}
Let $H$ be a separating, properly embedded surface in $M$. Let $V$ and $W$ be compressions on opposite sides of $H$. Then we say $(V,W)$ is a {\it weak reducing pair} for $H$ if $V \cap W=\emptyset$. 
\end{dfn}

\begin{dfn}
\label{d:TTStrongIrreducibility}
Let $H$ be a separating, properly embedded surface in $M$ which is not a torus. Then we say $H$ is {\it strongly irreducible} if there are compressions on opposite sides of $H$, but no weak reducing pairs. 
\end{dfn}

\begin{dfn}
\label{d:critical}
Let $H$ be a properly embedded, separating surface in $M$. The surface $H$ is {\it critical} if the compressions for $H$ can be partitioned into sets $C_0$ and $C_1$ such that:
\begin{enumerate}
	\item For each $i=0,1$ there is at least one pair of disks $V_i, W_i \in C_i$ such that $(V_i,W_i)$ is a weak reducing pair. 
	\item If $V \in C_0$ and $W \in C_1$ then $(V,W)$ is not a weak reducing pair.
\end{enumerate}
\end{dfn}

\section{Topologically minimal surfaces.}
\label{s:TopMinSfcs}

Topologically minimal surfaces generalize incompressible, strongly irreducible, and critical surfaces. In this section we review the definition of a topologically minimal surface, and its associated {\it topological index}, as given in \cite{TopIndexI}. 

Let $H$ be a properly embedded, separating surface with no torus components in a compact, orientable 3-manifold $M$. Then the {\it disk complex},  $\Gamma(H)$, is defined as follows:
	\begin{enumerate}
		\item  Vertices of $\Gamma(H)$ are isotopy classes of compressions for $H$. 
		\item A set of $m+1$ vertices forms an $m$-simplex if there are representatives for each that are pairwise disjoint. 
	\end{enumerate}

\begin{dfn}
\label{d:Indexn}
The {\it homotopy index} of a complex $\Gamma$ is defined to be 0 if $\Gamma=\emptyset$, and the smallest $n$ such that $\pi_{n-1}(\Gamma)$ is non-trivial, otherwise. If $\Gamma$ is contractible, its homotopy index is left undefined. We say a surface $H$ is {\it topologically minimal} if its disk complex $\Gamma(H)$ is either empty or non-contractible. When $H$ is topologically minimal, we say its {\it topological index} is the homotopy index of $\Gamma(H)$. 
\end{dfn}

In \cite{TopIndexI} we show that incompressible surfaces have topological index 0, strongly irreducible surfaces (see \cite{cg:87}) have topological index 1, and critical surfaces (see \cite{crit}) have topological index 2. In \cite{existence} we show that for each $n$ there is a manifold that contains a surface whose topological index is $n$.

\begin{thm}[\cite{TopIndexI}, Theorem 3.7.]
\label{t:OriginalIntersection}
Let $F$ be a properly embedded, incompressible surface in an irreducible 3-manifold $M$. Let $H$ be a properly embedded surface in $M$ with topological index $n$. Then $H$ may be isotoped so that
	\begin{enumerate}
		\item $H$ meets $F$ in $p$ saddles, for some $p \le n$. Away from these tangencies $H$ is transverse to $F$.
		\item The topological index of $H \setminus N(F)$ in $M \setminus N(F)$, plus $p$, is at most $n$. 
	\end{enumerate}
\end{thm}

In addition to this result about topological index, we will need the following:

\begin{lem}
\label{l:EssentialBoundary}
Suppose $H$ is a topologically minimal surface which is properly embedded in a 3-manifold $M$ with incompressible boundary. Then each loop of $\bdy H$ either bounds a component of $H$ that is a boundary-parallel disk, or is essential on $\bdy M$.
\end{lem}

\begin{proof}
Begin by removing from $H$ all components that are boundary-parallel disks. If nothing remains, then the result follows. Otherwise, the resulting surface (which we continue to call $H$) is still topologically minimal, as it has the same disk complex. Now, let $\alpha$ denote a loop of $\bdy H$ that is innermost among all such loops that are inessential on $\bdy M$ . Then $\alpha$ bounds a compression $D$ for $H$ that is disjoint from all other compressions. Hence, every maximal dimensional simplex of $\Gamma(H)$ includes the vertex corresponding to $D$. We conclude $\Gamma(H)$ is contractible to $D$, and thus $H$ was not topologically minimal. 
\end{proof}

\section{Topological index relative to boundaries.}
\label{s:IndexRelBdy}

In this section we define the topological index of a surface $H$ in a 3-manifold $M$ {\it with respect to $\bdy M$.} We then show that we may always obtain such a surface from a topologically minimal surface by a sequence of $\bdy$-compressions. 

Let $H$ be a properly embedded, separating surface with no torus components in a compact, orientable 3-manifold $M$. Then the complex $\Gamma(H;\bdy M)$, is defined as follows:
	\begin{enumerate}
		\item  Vertices of $\Gamma(H;\bdy M)$ are isotopy classes of compressions and $\bdy$-compressions for $H$. 
		\item A set of $m+1$ vertices forms an $m$-simplex if there are representatives for each that are pairwise disjoint. 
	\end{enumerate}

\begin{dfn}
\label{d:Indexn}
We say the {\it topological index of a surface $H$ with respect to $\bdy M$} is the homotopy index of the complex $\Gamma(H;\bdy M)$. If $H$ has a topological index with respect to $\bdy M$ then we say it is {\it topologically minimal with respect to $\bdy M$}. 
\end{dfn}

In Corollary 3.8 of \cite{TopIndexI} we showed that a topologically minimal surface can always be isotoped to meet an incompressible surface in a collection of essential loops. The exact same argument, with the words ``compression or $\bdy$-compression" substituted for ``compression" and ``innermost loop/outermost arc" substituted for ``innermost loop," shows:

\begin{lem}
\label{l:EssentialIntersection}
Let $M$ be a compact, orientable, irreducible 3-manifold with incompressible boundary. Let $H$ and $Q$ be properly embedded surfaces in $M$, where $H$ is topologically minimal with respect to $\bdy M$ and $Q$ is both incompressible and $\bdy$-incompressible. Then $H$ may be isotoped so that it meets $Q$ in a (possibly empty) collection of loops and arcs that are essential on both.  
\end{lem}

\begin{dfn}
\label{d:H/D}
If $D$ is a compression or $\bdy$-compression for a surface $H$ then we construct the surface $H/D$ as follows. Let $M(H)$ denote the manifold obtained from $M$ by cutting open along $H$. Let $B$ denote a neighborhood of $D$ in $M(H)$. The surface $H/D$ is obtained from $H$ by removing $B \cap H$ and replacing it with the frontier of $B$ in $M(H)$. The surface $H/D$ is said to have been obtained from $H$ by {\it compressing} or {\it $\bdy$-compressing} along $D$. Similarly, suppose $\tau$ is some simplex of $\Gamma(H;\bdy M)$ and $\{D_i\}$ are pairwise disjoint representatives of the vertices of $\tau$. Then $H/\tau$ is defined to be the surface obtained from $H$ by simultaneously compressing/$\bdy$-compressing along each disk $D_i$. 
\end{dfn}

We leave the proof of the following lemma as an exercise for the reader.

\begin{lem}
\label{l:CompressionEffect}
Suppose $M$ is an irreducible 3-manifold with incompressible boundary. Let $D$ be a $\bdy$-compression for a properly embedded surface $H$. Then $\Gamma(H/D)$ is the subset of $\Gamma(H)$ spanned by those compressions that are disjoint from $D$. \qed
\end{lem}

\begin{thm}
\label{t:MainTheorem}
Suppose $M$ is an irreducible 3-manifold with incompressible boundary.  Let $H$ be a surface whose topological index is $n$. Then either $H$ has topological index at most $n$ with respect to $\bdy M$, or there is a simplex $\tau$ of $\Gamma(H;\bdy M) \setminus \Gamma(H)$, such that the topological index of $\Gamma(H/\tau)$ is at most $n-{\rm dim}(\tau)$. 
\end{thm}

%Oops! Check that $H$ really does have top index n wrt K \cup F. Maybe I just know $\pi_{n-1}(\Gamma(H;K \cup F) \ne 1$, but don't know anything about lower homotopy groups. 

\begin{proof}
If $\Gamma(H) = \emptyset$ then the result is immediate, as any surface obtained by $\bdy$-compressing an incompressible surface must also be incompressible. 

If $\Gamma(H) \ne \emptyset$ then, by assumption, there is a non-trivial map $\iota$ from an $(n-1)$-sphere $S$ into the $(n-1)$-skeleton of  $\Gamma(H)$. Assuming the theorem is false will allow us to construct inductively a map  $\Psi$ of a $n$-ball $B$ into $\Gamma(H)$ such that $\Psi(\bdy B)=\iota(S)$. The existence of such a map contradicts the non-triviality of $\iota$. 

%For each vertex $v_i$ of $\Gamma(H)$ in the image of $\iota$, choose a representative $D_i$ such that if $\{v_i,v_j\}$ is a 1-simplex of $\Gamma(H)$, then $D_i \cap D_j=\emptyset$.

Note that $\Gamma(H) \subset \Gamma(H;\bdy M)$. If  $\pi_{n-1}(\Gamma(H;\bdy M)) \ne 1$ then the result is immediate. Otherwise,  $\iota$ can be extended to a map from  an $n$-ball $B$ into $\Gamma(H;\bdy M)$.

Let $\Sigma$ denote a triangulation of $B$ so that the map $\iota$ is simplicial. For each simplex $\tau$ of $\Sigma$ we let $\tau^{\bdy}$ denote the vertices of $\tau$ whose image under $\iota$ represent $\bdy$-compressions. If $\tau^{\bdy}=\emptyset$ then let $V_\tau=\Gamma(H)$.  Otherwise, let $V_\tau$ be the subspace of $\Gamma(H)$ spanned by the compressions that can be made disjoint from the disks represented by every vertex of $\iota(\tau^{\bdy})$. In other words, $V_\tau$ is the intersection of the link of $\tau^\bdy$ in $\Gamma(H;\bdy M)$ with $\Gamma(H)$. 

By Lemma \ref{l:CompressionEffect}, when $\tau^{\bdy} \ne \emptyset$ then $\Gamma(H/\tau^{\bdy})$ is precisely $V_\tau$. (More precisely, there is an embedding of $\Gamma(H/\tau^{\bdy})$ into $\Gamma(H)$ whose image is $V_\tau$.) By way of contradiction, we suppose the homotopy index of $\Gamma(H/\tau^{\bdy})$ is not less than or equal to  $n-{\rm dim}(\tau^{\bdy})$. Thus, $V_\tau \ne \emptyset$ and 
\begin{equation}
\label{e:contradictionRevisited}
\pi_i(V_\tau) =1 \mbox{ for all } i \le  n-1-{\rm dim}(\tau^{\bdy}). 
\end{equation}

\begin{clm}
\label{c:subsetRevisited}
Suppose $\tau$ is a cell of $\Sigma$ which lies on the boundary of a cell $\sigma$. Then $V_{\sigma} \subset V_{\tau}$. 
\end{clm}

\begin{proof}
Suppose $D \in V_{\sigma}$. Then $D$ can be made disjoint from disks represented by every vertex of $\iota(\sigma^{\bdy})$. Since $\tau$ lies on the boundary of $\sigma$, it follows that $\tau^{\bdy} \subset \sigma^{\bdy}$. Hence, $D$ can be made disjoint from the disks represented by every vertex of $\iota(\tau^{\bdy})$. It follows that $D \in V_\tau$. 
\end{proof}

Push the triangulation $\Sigma$ into the interior of $B$, so that $\rm{Nbhd}(\bdy B)$ is no longer triangulated (Figure \ref{f:SigmaDual}(b)). Then extend $\Sigma$ to a cell decomposition over all of $B$ by forming the product of each cell of $\Sigma \cap S$ with the interval $I$ (Figure \ref{f:SigmaDual}(c)). Denote this cell decomposition as $\Sigma'$. Note that $\iota$ extends naturally over $\Sigma'$, and the conclusion of Claim \ref{c:subsetRevisited} holds for cells of $\Sigma'$. Now let $\Sigma^*$ denote the dual cell decomposition of $\Sigma'$ (Figure \ref{f:SigmaDual}(d)). This is done in the usual way, so that there is a correspondence between the $d$-cells of $\Sigma ^*$ and the $(n-d)$-cells of $\Sigma'$. Note that, as in the figure, $\Sigma ^*$ is not a cell decomposition of all of $B$, but rather a slightly smaller $n$-ball, which we call $B'$. 

\begin{figure}
\psfrag{a}{(a)}
\psfrag{b}{(b)}
\psfrag{c}{(c)}
\psfrag{d}{(d)}
\begin{center}
\includegraphics[width=5 in]{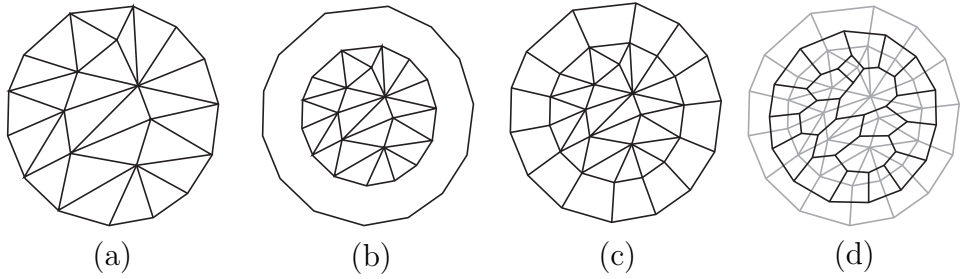}
\caption{(a) The triangulation $\Sigma$ of $B$. (b) Push $\Sigma$ into the interior of $B$. (c) Fill in $\rm{Nbhd}(\bdy B)$ with product cells to complete $\Sigma'$. (d) $\Sigma^*$ is the dual of $\Sigma'$.}
\label{f:SigmaDual}
\end{center}
\end{figure}

For each cell $\tau$ of $\Sigma'$, let $\tau^*$ denote its dual in $\Sigma^*$. Thus, it follows from Claim \ref{c:subsetRevisited} that if $\sigma^*$ is a cell of $\Sigma^*$ that is on the boundary of $\tau^*$, then $V_{\sigma} \subset V_\tau$. 

We now produce a contradiction by defining a continuous map $\Psi:B' \to \Gamma(H)$ such that $\Psi(\bdy B')=\iota(S)$. The map will be defined inductively on the $d$-skeleton of $\Sigma^*$ so that the image of every cell $\tau^*$ is contained in $V_\tau$. 

For each $0$-cell $\tau^* \in \Sigma^*$, choose a point in $V_\tau$ to be $\Psi(\tau^*)$. If $\tau^*$ is in the interior of $B'$ then this point may be chosen arbitrarily in $V_\tau$. If $\tau^* \in \bdy B'$ then $\tau$ is an $n$-cell of $\Sigma'$. This $n$-cell is $\sigma \times I$, for some $(n-1)$-cell $\sigma$ of $\Sigma \cap S$. But since $\iota(S) \subset \Gamma(H)$, it follows that $\tau^{\bdy}=\sigma ^{\bdy}=\emptyset$, and thus $V_\tau=\Gamma(H)$. We conclude $\iota(\tau) \subset V_\tau$, and thus we can choose $\tau^*$, the barycenter of $\iota(\tau)$, to be $\Psi(\tau^*)$. 

We now proceed to define the rest of the map $\Psi$ by induction. Let $\tau^*$ be a $d$-dimensional cell of $\Sigma^*$ and assume $\Psi$ has been defined on the $(d-1)$-skeleton of $\Sigma^*$. In particular, $\Psi$ has been defined on $\bdy \tau^*$. Suppose $\sigma^*$ is a cell on $\bdy \tau^*$.  By Claim \ref{c:subsetRevisited} $V_\sigma \subset V_\tau$. By assumption $\Psi|\sigma^*$ is defined and $\Psi(\sigma^*) \subset V_\sigma$. We conclude $\Psi(\sigma^*) \subset V_\tau$ for all $\sigma^* \subset \bdy \tau^*$, and thus
\begin{equation}
\label{e:boundary}
\Psi(\bdy \tau^*) \subset V_\tau.
\end{equation}

Note that  \[{\rm dim}(\tau) = n-{\rm dim}(\tau^*)=n-d.\] Since ${\rm dim}(\tau^{\bdy}) \le {\rm dim}(\tau)$, we have \[{\rm dim}(\tau^{\bdy}) \le n-d.\] Thus \[d \le n-{\rm dim}(\tau^{\bdy}),\] and finally \[d-1 \le n-1-{\rm dim}(\tau^{\bdy}).\]
It now follows from Equation \ref{e:contradictionRevisited} that $\pi_{(d-1)}(V_\tau)=1$. Since $d-1$ is the dimension of $\bdy \tau^*$, we can thus extend $\Psi$ to a map from $\tau^*$ into $V_\tau$. 

Finally, we claim that if $\tau^* \subset \bdy B'$ then this extension of $\Psi$ over $\tau^*$ can be done in such a way so that $\Psi(\tau^*) =\iota(\tau^*)$. This is because in this case each vertex of $\iota(\tau)$ is a compression, and hence $V_\tau=\Gamma(H)$. As $\iota(S) \subset \Gamma(H)=V_\tau$, we have $\iota(\tau^*) \subset V_\tau$. Thus we may choose $\Psi(\tau^*)$ to be $\iota(\tau^*)$.
\end{proof}

\begin{cor}
\label{c:FirstDistanceBound}
Suppose $M$ is a compact irreducible 3-manifold with incompressible boundary. Let $F$ denote a component of $\bdy M$. Let $H$ be a surface whose topological index is $n$. Then either $H$ is isotopic into a neighborhood of $\bdy M$ or there is a surface $H'$, obtained from $H$ by some sequence of (possibly simultaneous) $\bdy$-compressions, such that 
	\begin{enumerate}
		\item $H'$ has topological index at most $n$ with respect to $\bdy M$, and 
		\item The distance between $H \cap F$ and $H' \cap F$ is at most $3\chi(H')-3\chi(H)$.
	\end{enumerate}
\end{cor}

Here we are measuring the {\it distance} between loops on $F$ by the path distance in either the Farey graph or the curve complex, depending on whether or not $F$ is a torus.

\begin{proof}
We first employ Theorem \ref{t:MainTheorem} to obtain a sequence of surfaces, $\{H_i\}_{0=1}^m$ in $M$, such that 
	\begin{enumerate}
		\item $H_0=H$
		\item $H_{i+1}=H_i/\tau_i$, for some simplex $\tau_i \subset \Gamma(H_i; \bdy M) \setminus \Gamma(H_i)$.
		\item For each $i$ the topological index of $H_{i+1}$ in $M$ is $n_{i+1} \le n_i-\rm{dim}(\tau_i)$, where $n_0=n$.
		\item $H_m$ has topological index at most $n_m$ with respect to $\bdy M$. 
	\end{enumerate}

It follows from Lemma \ref{l:EssentialBoundary} that $\bdy H_i$ contains at least one component that is essential on $\bdy M$ for each $i<m$. There is a component of the boundary of $H_m$ that is essential if and only if $H_m$ is not a collection of $\bdy$-parallel disks. However, if $H_m$ is a collection of $\bdy$-parallel disks, then the surface $H$ was isotopic into a neighborhood of $\bdy M$. Hence, it suffices to show that for all $i$, the distance between the loops of $H_i \cap F$ and $H_{i+1} \cap F$ is bounded by $3\chi(H_{i+1})-3\chi(H_i)$.

Let $\VV$ and $\WW$ denote the sides of $H_i$ in $M$. If the dimension of $\tau_i$ is $k-1$, then $H_{i+1}$ is obtained from $H_i$ by $k$ simultaneous $\bdy$-compressions. Hence, the difference between the Euler characteristics of $H_i$ and $H_{i+1}$ is precisely $k$. 

If $k=1$ then the loops of $\bdy H_i$ can be made disjoint from the loops of $\bdy H_{i+1}$. It follows that 
\[d(H_i \cap F, H_{i+1} \cap F) \le 1=k < 3k,\]
and thus the result follows. Henceforth, we will assume $k \ge 2$. 

Let $\{V_1,...,V_p\}$ denote the $\bdy$-compressions represented by vertices of $\tau_i$ that lie in $\VV$, and $\{W_1,...,W_q\}$ the $\bdy$-compressions represented by vertices of $\tau_i$ that lie in $\WW$. We will assume that $\VV$ and $\WW$ were labelled so that $p \le k/2$. The loops of $H_{i+1}\cap F$ are obtained from the loops of $H_i \cap F$ by band sums along the arcs of $V_i \cap F$ and $W_i \cap F$. By pushing the loops of $H_i \cap F$ slightly into $\VV$, we see that they meet the loops of $H_{i+1} \cap F$ at most $4p$ times. That is, \[|H_i \cap H_{i+1}\cap F| \le 4p \le 2k.\]

When $F$ is not a torus we measure the distance between curves $\alpha$ and $\beta$ in $F$ in its curve complex. By \cite{hempel:01}, this distance is bounded above by $2+2\log _2 |\alpha \cap \beta|$. Hence, the distance between $H_i \cap F$ and $H_{i+1} \cap F$ is at most $2+2\log_2 2k$. But for any positive integer $k$, $\log _2 2k \le k$. Hence, we have shown $d(H_i \cap F , H_{i+1} \cap F) \le 2+2k$. Since $k \ge 2$ it follows that  $2+2k \le 3k$, and thus the result follows.

When $F \cong T^2$, the distance between curves $\alpha$ and $\beta$ is measured in the Farey graph. In this case their distance is bounded above by $1+\log _2 |\alpha \cap \beta|$. As this bound is twice as good as before, the distance between $H_i \cap F$ and $H_{i+1}\cap F$ must satisfy the same inequality. 
\end{proof}

\section{Complicated amalgamations}
\label{s:DistanceBoundSection}

The results of this section are due to T. Li when the index of $H$ is zero or one \cite{li:08}. The arguments presented here for the more general statements borrow greatly from these ideas.

\begin{lem}
\label{l:EulerBound}
\cite{li:08} Let $M$ be a compact, orientable, irreducible 3-manifold with incompressible boundary. Suppose $H$ and $Q$ are properly embedded surfaces in $M$ that are both incident to some component $F$ of $\bdy M$, where $Q$ is incompressible and $\bdy$-incompressible, and every loop and arc of $H \cap Q$ is essential on both surfaces. Then one of the following holds:
	\begin{enumerate}
		\item There is an incompressible, $\bdy$-incompressible surface $Q'$ that meets $H$ in fewer arcs than $Q$ did. The surface $Q'$ is either isotopic to $Q$, or is an annulus incident to $F$. 
		\item The number of arcs in $H \cap Q$ that are incident to $F$ is at most $(1-3\chi(H))(1-3\chi(Q))$. 
	\end{enumerate}
\end{lem}

\begin{proof}
If $H$ is not an annulus, then there can be at most $-3\chi(H)$ arcs in any collection of non-parallel essential arcs on $H$. When $H$ is an annulus, then there can be at most one arc in such a collection. Hence, in either case there can be at most $1-3\chi(H)$ non-parallel essential arcs. Similarly, there can be at most $1-3\chi(Q)$ arcs in any collection of non-parallel essential arcs on $Q$. Hence, if the number of arcs in $H \cap Q$ incident to $F$ is larger than $(1-3\chi(H))(1-3\chi(Q))$, then there must be at least two arcs $\alpha$ and $\beta$ of $H \cap Q$, incident to $F$, that are parallel on both $H$ and $Q$.  Suppose this is the case, and let $R_H$ and $R_Q$ denote the rectangles cobounded by $\alpha$ and $\beta$ on $H$ and $Q$, respectively. Note that $\alpha$ and $\beta$ can be chosen so that $A=R_H \cup R_Q$ is an embedded annulus. 

Since $\alpha$ is essential on $Q$ and $Q$ is $\bdy$-incompressible, it follows that $\alpha$ is essential in $M$. As $\alpha$ is also contained in $A$, we conclude $A$ is $\bdy$-incompressible. If $A$ is also incompressible then the result follows, as $A$ meets $H$ in fewer arcs than $Q$, and $A$ is incident to $F$. 

If $A$ is compressible then it must bound a 1-handle $V$ in $M$, since it contains an arc that is essential in $M$. The 1-handle $V$ can be used to guide an isotopy of $Q$ that takes $R_Q$ to $R_H$. This isotopy may remove other components of $Q \cap V$ as well. The resulting surface $Q'$ meets $H$ in at least two fewer arcs that are incident to $F$. 
\end{proof}

\begin{lem}
\label{l:SecondDistanceBound}
Let $M$ be a compact, orientable, irreducible 3-manifold with incompressible boundary, which is not an $I$-bundle. Suppose $H$ and $Q$ are properly embedded surfaces in $M$ that are both incident to some component $F$ of $\bdy M$, where $H$ is topologically minimal with respect to $\bdy M$, and $Q$ is an incompressible, $\bdy$-incompressible surface with maximal Euler characteristic. Then the distance between $H \cap F$ and $Q \cap F$ is at most $4+2(1-3\chi(H))(1-3\chi(Q))$.
\end{lem}

\begin{proof}
By Lemma \ref{l:EssentialIntersection} $H$ and $Q$ can be isotoped so that they meet in a collection of loops and arcs that are essential on both surfaces. Assume first $Q$ is not an annulus, and it meets $H$ in the least possible number of essential arcs. Then by Lemma \ref{l:EulerBound}, the number of arcs in $H \cap Q$ incident to $F$ is at most $(1-3\chi(H))(1-3\chi(Q))$. As each such arc has at most two endpoints on $F$, \[|H \cap Q \cap F| \le 2(1-3\chi(H))(1-3\chi(Q)).\] 

As in the proof of Corollary \ref{c:FirstDistanceBound}, when $F$ is not a torus we measure the distance between curves $\alpha$ and $\beta$ in $F$ in its curve complex. By \cite{hempel:01}, this distance is bounded above by $2+2\log _2 |\alpha \cap \beta|$. So we have, \[d(H \cap F, Q \cap F) \le 2+2\log _2 2(1-3\chi(H))(1-3\chi(Q)).\] But for any positive integer $n$, $\log _2 2n \le n$. Hence,
\begin{eqnarray*}
d(H \cap F, Q \cap F) &\le& 2+2(1-3\chi(H))(1-3\chi(Q))\\
&\le & 4+2(1-3\chi(H))(1-3\chi(Q)).
\end{eqnarray*}

When $F \cong T^2$, the distance between curves $\alpha$ and $\beta$ is measured in the Farey graph. In this case their distance is bounded above by $1+\log _2 |\alpha \cap \beta|$. As this bound is twice as good as before, the distance between $H \cap F$ and $Q \cap F$ must satisfy the same inequality. 

If $Q$ is an annulus, then since $M$ is not an $I$-bundle we may apply  Lemma 3.2 of \cite{li:08}, which says that $Q \cap F$ is at most distance 2 away from $Q' \cap F$, for some incompressible, $\bdy$-incompressible annulus $Q'$ that meets $H$ in the least possible number of essential arcs. By Lemma \ref{l:EulerBound}, the number of arcs in $H \cap Q'$ incident to $F$ is at most $(1-3\chi(H))(1-3\chi(Q'))$. As above, this implies the distance between $H \cap F$ and $Q' \cap F$ is at most \[2+2(1-3\chi(H))(1-3\chi(Q')).\] It follows that the distance between $H \cap F$ and $Q \cap F$ is at most $4+2(1-3\chi(H))(1-3\chi(Q))$. 
\end{proof}

\begin{lem}
\label{l:ThirdDistanceBound}
Let $M$ be a compact, orientable, irreducible 3-manifold with incompressible boundary, which is not an $I$-bundle. Suppose $H$ and $Q$ are properly embedded surfaces in $M$ that are both incident to some component $F$ of $\bdy M$, where $H$ is topologically minimal and $Q$ is an incompressible, $\bdy$-incompressible surface with maximal Euler characteristic. Then either $H$ is isotopic into a neighborhood of $\bdy M$, or the distance between $H \cap F$ and $Q \cap F$ is at most $4+2(1-3\chi(H))(1-3\chi(Q))$.
\end{lem}

\begin{proof}
If $H$ is not isotopic into a neighborhood of $\bdy M$ then by Corollary \ref{c:FirstDistanceBound} we may obtain a surface $H'$ from $H$ by a sequence of $\bdy$-compressions, which is topologically minimal with respect to $\bdy M$, where \[d(H \cap F,H' \cap F) \le 3\chi(H')-3\chi(H).\]

By Lemma \ref{l:SecondDistanceBound}, 
\[d(H' \cap F, Q \cap F) \le 4+2(1-3\chi(H'))(1-3\chi(Q)).\]

Putting these together gives:
\begin{eqnarray*}
d(H \cap F,Q \cap F) & \le & d(H \cap F, H' \cap F) + d(H' \cap F, Q \cap F)\\
& \le & 3\chi(H')-3\chi(H) + 4+2(1-3\chi(H'))(1-3\chi(Q))\\
&=&6-3\chi(H) -3\chi(H')-6\chi(Q)+18\chi(H')\chi(Q)\\
&\le&6-3\chi(H) -3\chi(H)-6\chi(Q)+18\chi(H)\chi(Q)\\
&=&6-6\chi(H)-6\chi(Q)+18\chi(H)\chi(Q)\\
&=&4+2(1-3\chi(H))(1-3\chi(Q))
\end{eqnarray*}
\end{proof}

\begin{thm}
\label{t:DistanceBoundTheorem}
Let $X$  be a compact, orientable (not necessarily connected), irreducible 3-manifold with incompressible boundary, such that no component of $X$ is an $I$-bundle. Suppose some components $F_-$ and $F_+$ of $\bdy X$ are homeomorphic. Let $Q$ denote an incompressible, $\bdy$-incompressible (not necessarily connected) surface in $X$  of maximal Euler characteristic that is incident to both $F_-$ and $F_+$. Let $K=24(1-3\chi(Q))$. Suppose $\phi:F_- \to F_+$ is a gluing map such that \[d(\phi(Q \cap F_-), Q \cap F_+) \ge Kg.\] Let $M$ denote the manifold obtained from $X$ by gluing $F_-$ to $F_+$ via the map $\phi$. Let $F$ denote the image of $F_-$ in $M$. Then any closed, topologically minimal surface in $M$ whose genus is at most $g$ can be isotoped to be disjoint from $F$.
\end{thm}

\begin{proof}
Suppose $H$ is a topologically minimal surface in $M$ whose genus is at most $g$. By Theorem \ref{t:OriginalIntersection} $H$ may be isotoped so that it meets $F$ transversally away from a collection of saddles, and so that the components of $H \setminus N(F)$ are topologically minimal in $M \setminus N(F)$. Note that $M \setminus N(F) = X'$, where $X' \cong X$. We denote the images of $F_-$ and $F_+$ in $X'$ by the same names. Let $H'=H \cap X'$. By Lemma \ref{l:EssentialBoundary}, $\bdy H'$ consists of essential loops on $\bdy X'$. When projected to $F$, these loops are all on the boundary of a neighborhood of the 4-valent graph $H \cap F$. It follows that the distance between $H' \cap F_-$ and $H' \cap F_+$ is at most one.

If $H$ could not have been isotoped to be disjoint from $F$, then the surface $H'$ can not be isotopic into a neighborhood of $\bdy X'$. We may thus apply Lemma \ref{l:ThirdDistanceBound} to obtain:

\begin{eqnarray*}
d(Q \cap F_-, Q \cap F_+) & \le & d(Q \cap F_-, H' \cap F_-)+ d(H' \cap F_-, H' \cap F_+)\\
&&+d(Q \cap F_+, H' \cap F_+)\\
&\le& 2(4+2(1-3\chi(H'))(1-3\chi(Q)))+1\\
& = & 4(1-3\chi(H'))(1-3\chi(Q))+9\\
& \le & 4(1-3\chi(H))(1-3\chi(Q))+9\\
&\le &4(1-3\chi(H))(1-3\chi(Q))+9(1-3\chi(Q))\\
&=&(13-12\chi(H))(1-3\chi(Q))\\
&\le&(24g-11)(1-3\chi(Q))\\
&\le&24g(1-3\chi(Q))
\end{eqnarray*}
\end{proof}

Theorem \ref{t:DistanceBoundTheorem} motivates us to make the following definition:

\begin{dfn}
An incompressible surface $F$ in a 3-manifold $M$ is a {\it $g$-barrier surface} if any incompressible, strongly irreducible, or critical surface in $M$ whose genus is at most $g$ can be isotoped to be disjoint from $F$. 
\end{dfn}

By employing Theorem \ref{t:DistanceBoundTheorem} we may construct 3-manifolds with any number of $g$-barrier surfaces. Simply begin with a collection of 3-manifolds and successively glue boundary components together by ``sufficiently complicated" maps.

\section{Generalized Heegaard Splittings}
\label{s:GHS}

In this section we define {\it Heegaard splittings} and {\it Generalized Heegaard Splittings}. The latter structures were first introduced by Scharlemann and Thompson \cite{st:94} as a way of keeping track of handle structures. The definition we give here is more consistent with the usage in \cite{gordon}.

\begin{dfn}
A {\it compression body} $\mathcal C$ is a manifold formed in one of the following two ways:
	\begin{enumerate}
		\item Starting with a 0-handle, attach some number of 1-handles. In this case we say $\bdy _- \mathcal C=\emptyset$ and $\bdy _+ \mathcal C=\bdy \mathcal C$. 
		\item Start with some (possibly disconnected) surface $F$ such that each component has positive genus. Form the product $F \times I$. Then attach some number of 1-handles to $F \times \{1\}$. We say $\bdy _- \mathcal C= F \times \{0\}$ and $\bdy _+ \mathcal C$ is the rest of $\bdy \mathcal C$. 
	\end{enumerate}
\end{dfn}

\begin{dfn}
\label{d:Heegaard}
Let $H$ be a properly embedded, transversally oriented surface in a 3-manifold $M$, and suppose $H$ separates $M$ into $\VV$ and $\WW$. If $\VV$ and $\WW$ are compression bodies and $\VV \cap \WW=\partial _+ \VV=\partial _+ \WW=H$, then we say $H$ is a {\it Heegaard surface} in $M$.
\end{dfn}

\begin{dfn}
The transverse orientation on the Heegaard surface $H$ in the previous definition is given by a choice of normal vector. If this vector points into $\VV$, then we say any subset of $\VV$ is {\it above} $H$ and any subset of $\WW$ is {\it below} $H$. 
\end{dfn}

%Not sure where to put this:
\begin{dfn}
Suppose $H$ is a Heegaard surface in a manifold $M$ with non-empty boundary. Let $F$ denote a component of $\bdy M$. Then the surface $H'$ obtained from $H$ by attaching a copy of $F$ to it by an unknotted tube is also a Heegaard surface in $M$. We say $H'$ was obtained from $H$ by a {\it boundary-stabilization along $F$}. The reverse operation is called a {\it boundary-destabilization} along $F$. 
\end{dfn}

\begin{dfn}
\label{d:GHS}
A {\it generalized Heegaard splitting (GHS)} $H$ of a 3-manifold $M$ is a pair of sets of transversally oriented, connected, properly embedded surfaces,  $\thick{H}$ and $\thin{H}$ (called the {\it thick levels} and {\it thin levels}, respectively), which satisfy the following conditions. 
	\begin{enumerate}
		\item Each component $M'$ of $M \setminus \thin{H}$ meets a unique element $H_+$ of $\thick{H}$. The surface $H_+$ is a Heegaard surface in $\overline{M'}$ dividing $\overline{M'}$ into compression bodies $\VV$ and $\WW$. Each component of $\bdy _- \VV$ and $\bdy _- \WW$ is an element of $\thin{H}$. Henceforth we will denote the closure of the component of $M \setminus \thin{H}$ that contains an element $H_+ \in \thick{H}$ as $M(H_+)$. 
		\item Suppose $H_-\in \thin{H}$. Let $M(H_+)$ and $M(H_+')$ be the submanifolds on each side of $H_-$. Then $H_-$ is below $H_+$ in $M(H_+)$ if and only if it is above $H_+'$ in $M(H_+')$.
		\item The term ``above" extends to a partial ordering on the elements of $\thin{H}$ defined as follows. If $H_-$ and $H'_-$ are subsets of $\bdy M(H_+)$, where $H_-$ is above $H_+$ in $M(H_+)$ and $H_-'$ is below $H_+$ in $M(H_+)$, then $H_-$ is above $H_-'$ in $M$.
	\end{enumerate}
\end{dfn}

\begin{dfn}
Suppose $H$ is a GHS of an irreducible 3-manifold $M$. Then $H$ is {\it strongly irreducible} if each element $H_+ \in \thick{H}$ is strongly irreducible in $M(H_+)$. The GHS $H$ is {\it critical} if each element $H_+ \in \thick{H}$ but exactly one is strongly irreducible in $M(H_+)$, and the remaining element is critical in $M(H_+)$. 
\end{dfn}

The strongly irreducible case of the following result is due to Scharlemann and Thompson \cite{st:94}. The proof in the critical case is similar. 

\begin{thm}
\label{t:IncompressibleThinLevels}{\rm (\cite{gordon}, Lemma 4.6)} 
Suppose $H$ is a strongly irreducible or critical GHS of an irreducible 3-manifold $M$. Then each thin level of $H$ is incompressible.
\end{thm}

\section{Reducing GHSs}
\label{s:ReducingGHS}

\begin{dfn}
Let $H$ be an embedded surface in $M$.  Let $D$ be a compression for $H$. Let $\VV$ denote the closure of the component of $M \setminus H$ that contains $D$. (If $H$ is non-separating then $\VV$ is the manifold obtained from $M$ by cutting open along $H$.) Let $N$ denote a regular neighborhood of $D$ in $\VV$. To {\it surger} or {\it compress} $H$ along $D$ is to remove $N \cap H$ from $H$ and replace it with the frontier of $N$ in $\VV$. We denote the resulting surface by $H/D$. 
\end{dfn}

It is not difficult to find a complexity for surfaces which decreases under compression. Incompressible surfaces then represent ``local minima" with respect to this complexity. We now present an operation that one can perform on GHSs that also reduces some complexity (see Lemma 5.14 of \cite{gordon}).  Strongly irreducible GHSs will then represent ``local minima" with respect to such a complexity. This operation is called {\it weak reduction}.

\begin{dfn}
Let $H$ be a properly embedded surface in $M$. If $(D,E)$ is a weak reducing pair for $H$, then we let $H/DE$ denote the result of simultaneous surgery along $D$ and $E$.
\end{dfn}

\begin{dfn}
\label{d:PreWeakReduction}
Let $M$ be a compact, connected, orientable 3-manifold. Let $G$ be a GHS. Let $(D,E)$ be a weak reducing pair for some  $G_+ \in \thick{G}$. Define 
	\[T(H)=\thick{G} -\{G_+\} \cup \{G_+ / D, G_+ / E\},\ \mbox{and}\]
	\[t(H) =\thin{G} \cup \{G_+/DE\}.\]

A new GHS $H=\{\thick{H},\thin{H}\}$ is then obtained from $\{T(H), t(H)\}$ by successively removing the following:
	\begin{enumerate}
		\item Any sphere element $S$ of $T(H)$ or $t(H)$ that is inessential, along with any elements of $t(H)$ and $T(H)$ that lie in the ball that it (co)bounds. 
		\item Any element $S$ of $T(H)$ or $t(H)$ that is $\bdy$-parallel, along with any elements of $t(H)$ and $T(H)$ that lie between $S$ and $\bdy M$. 
		\item Any elements $H_+ \in T(H)$ and $H_- \in t(H)$, where $H_+$ and $H_-$ cobound a submanifold $P$ of $M$, such that $P$ is a product, $P \cap T(H)=H_+$, and $P \cap t(H)=H_-$. 
	\end{enumerate}

We say the GHS $H$ is obtained from $G$ by {\it weak reduction} along $(D,E)$.
\end{dfn}

The first step in weak reduction is illustrated in Figure \ref{f:WeakReduction}. 

        \begin{figure}[htbp]
        \psfrag{1}{$G_+/D$}
        \psfrag{2}{$G_+/E$}
        \psfrag{3}{$G_+/DE$}
        \psfrag{G}{$G_+$}
        \psfrag{E}{$E$}
        \psfrag{D}{$D$}
        \vspace{0 in}
        \begin{center}
       \includegraphics[width=3.5 in]{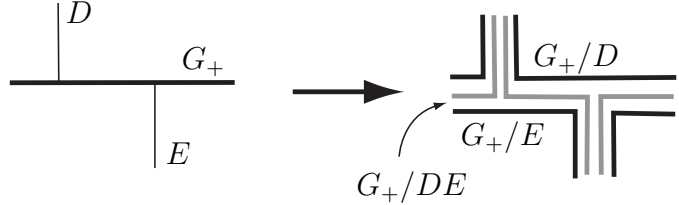}
       \caption{The first step in weak reduction.}
        \label{f:WeakReduction}
        \end{center}
        \end{figure}

\begin{dfn}
\label{d:destabilization}
The weak reduction of a GHS given by the weak reducing pair $(D,E)$ for the thick level $G_+$ is called a {\it destabilization} if $G_+/DE$ contains a sphere.
\end{dfn} 

In the next section we give a coarse measure of complexity for GHSs called {\it genus}. Destabilizations are precisely those weak reductions that reduce genus.

\section{Amalgamations}
\label{s:Amalgamations}

Let $H$ be a GHS of a connected 3-manifold $M$. In this section we use $H$ to produce a complex that is the spine of a Heegaard surface in $M$. We call this Heegaard surface the {\it amalgamation} of $H$. Most of this material is reproduced from \cite{gordon}. First, we must introduce some new notation. 

\begin{dfn}
Let $H$ be a Heegaard surface in $M$. Let $\Sigma$ denote a properly embedded graph in $M$. Let $(\bdy M) '$ denote the union of the boundary components of $M$ that meet $\Sigma$. Then we say $(\bdy M)' \cup \Sigma$ is a {\it spine} of $H$ if the frontier of a neighborhood of $(\bdy M)' \cup \Sigma$ is isotopic to $H$.
\end{dfn}

Suppose $H$ is a GHS of $M$ and $H_{+}\in \thick{H}$. Recall that $H_{+}$ is transversely oriented, so that we may consistently talk about those points of $M(H_+)$ that are ``above" $H_+$ and those points that are ``below." The surface $H_+$ divides $M(H_+)$ into two compression bodies. Henceforth we will denote these compression bodies as $\CV(H_+)$ and $\CW(H_+)$, where $\CV(H_+)$ is below $H_+$ and $\CW(H_+)$ is above. When we wish to make reference to an arbitrary compression body which lies above or below some thick level we will use the notation $\CV$ and $\CW$. Define $\partial _- M(H_+)$ to be $\partial _- \CV(H_+)$ and $\partial _+ M(H_+)$ to be $\partial _- \CW (H_+)$. That is, $\partial _- M(H_+)$ and $\partial _+ M(H_+)$ are the boundary components of $M(H_+)$ that are below and above $H_+$, respectively. If $N$ is a union of manifolds of the form $M(H_i)$ for some set of thick levels $\{H_i\} \subset \thick{H}$ then we let $\bdy _{\pm} N$ denote the union of those boundary components of $N$ that are components of $\bdy _{\pm} M(H_i)$, for some $i$. 

We now define a sequence of manifolds $\{M_i\}$ where
\[M_0 \subset M_1 \subset ... \subset M_n=M.\]
The submanifold $M_0$ is defined to be the disjoint union of all manifolds of the form $M(H_+)$, such that  $\partial _- M(H_+) \subset \bdy M$. (In particular, if $\bdy _-M(H_+)=\emptyset$ for some $H_+ \in \thick{H}$, then $M(H_+)\subset M_0$.) The fact that the thin levels of $H$ are partially ordered guarantees $M_0 \ne \emptyset$. Now, for each $i$ we define $M_i$ to be the union of $M_{i-1}$ and all manifolds $M(H_+)$ such that $\partial _- M(H_+) \subset \partial M_{i-1} \cup \bdy M$. Again, it follows from the partial ordering of thin levels that for some $i$ the manifold $M_i=M$. 

We now define a sequence of complexes $\Sigma_i$ in $M$. The final element of this sequence will be a complex $\Sigma$. This complex will be a spine of the desired Heegaard surface. The intersection of $\Sigma$ with some $M(H_+)$ is depicted in Figure \ref{f:Amalgam}.

Each $\CV \subset M_0$ is a compression body. Choose a spine of each, and let $\Sigma'_0$ denote the union of these spines. The complement of $\Sigma'_0$ in $M_0$ is a (disconnected) compression body, homeomorphic to the union of the compression bodies $\CW \subset M_0$. Now let $\Sigma_0$ be the union of $\Sigma'_0$ and one vertical arc for each component $H_-$ of $\partial _+ M_0$,  connecting $H_-$ to $\Sigma' _0$. 

We now assume $\Sigma _{i-1}$ has been constructed and we construct $\Sigma_i$. Let $M_i'=\overline{M_i - M_{i-1}}$. For each compression body $\CV \subset M_i'$ choose a set of arcs $\Gamma \subset \CV$ such that $\partial \Gamma \subset \Sigma _{i-1} \cap \partial M_{i-1}$, and such that the complement of $\Gamma$ in $\CV$ is a product. Let $\Sigma' _i$ be the union of $\Sigma _{i-1}$ with all such arcs $\Gamma$, and all components of $\bdy _- \CV$ that are contained in $\bdy M$. Now let $\Sigma_i$ be the union of $\Sigma'_i$ and one vertical arc for each component $H_-$ of $\partial _+ M_i$, connecting $H_-$ to $\Sigma' _i$. 

        \begin{figure}[htbp]
        \psfrag{h}{$H_+$}
        \psfrag{W}{$\CV(H_+)$}
        \psfrag{w}{$\CW(H_+)$}
        \psfrag{S}{$\Sigma$}
        \vspace{0 in}
        \begin{center}
       \includegraphics[width=3 in]{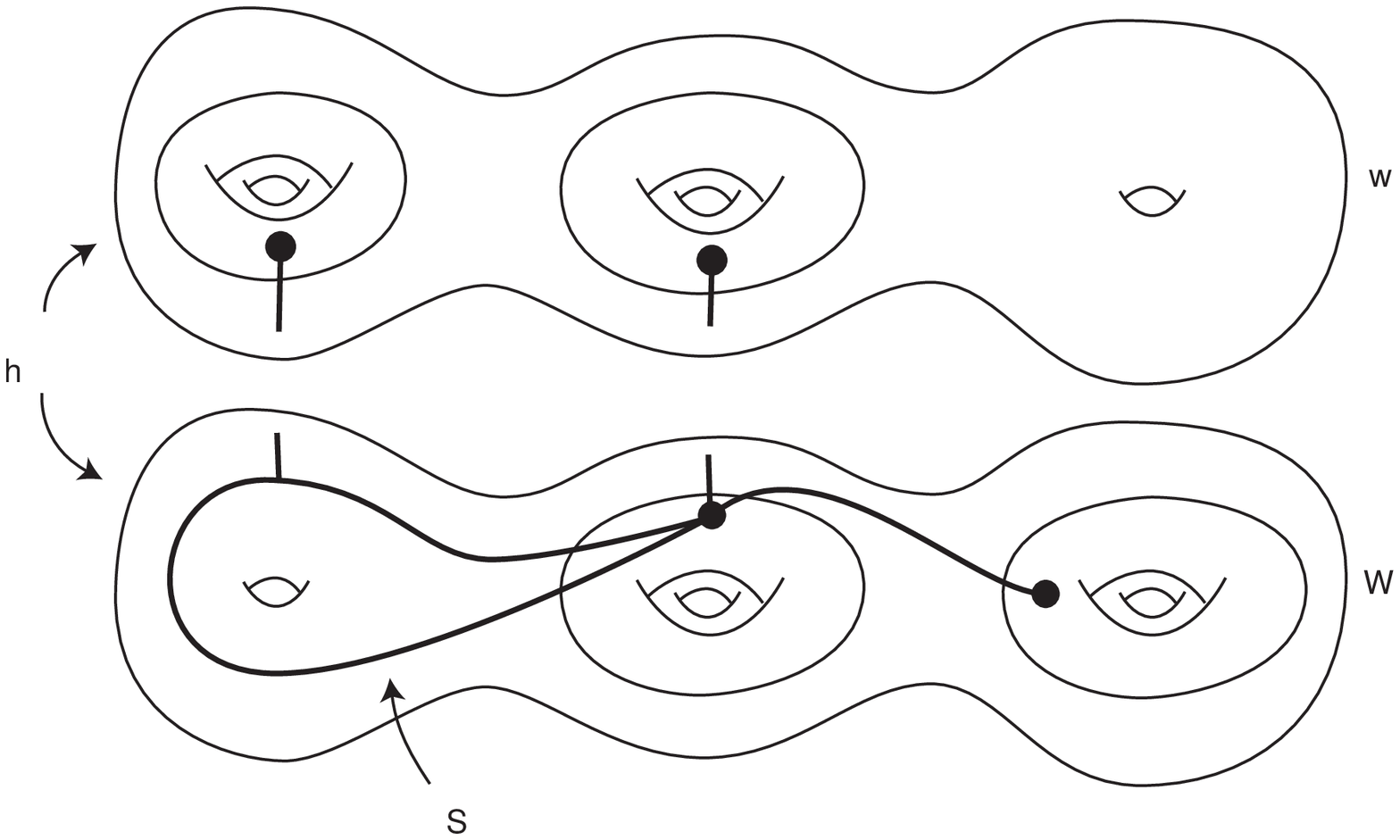}
       \caption{The intersection of $\Sigma$ with $\CV(H_+)$ and $\CW(H_+)$.}
        \label{f:Amalgam}
        \end{center}
        \end{figure}

\begin{lem}
{\rm (\cite{gordon}, Lemma 7.2)}  If $H$ is a GHS of $M$ then the complex $\Sigma$ defined above is the spine of a Heegaard surface in $M$.
\end{lem}

\begin{dfn}
\label{d:amalgam}
Let $H$ be a GHS and $\Sigma$ be the complex in $M$ defined above. The Heegaard surface that $\Sigma$ is a spine of is called the {\it amalgamation} of $H$ and will be denoted $\amlg{H}$.
\end{dfn}

Note that although the construction of the complex $\Sigma$ involved some choices, its neighborhood is uniquely defined up to isotopy at each stage. Hence, the amalgamation of a GHS is well defined, up to isotopy. 

For the next lemma, recall the definition of {\it destabilization}, given in Definition \ref{d:destabilization}.

\begin{lem}
\label{l:AmalgGenus}
{\rm (\cite{gordon}, Corollary 7.5)}  Suppose $M$ is irreducible, $H$ is a GHS of $M$ and $G$ is obtained from $H$ by a weak reduction which is not a destabilization. Then $\amlg{H}$ is isotopic to $\amlg{G}$.
\end{lem}

It follows that if a GHS $G$ is obtained from a GHS $H$ by a weak reduction or a destabilization then the genus of $\amlg{G}$ is at most the genus of $\amlg{H}$.

\begin{dfn}
The {\it genus} of a GHS is the genus of its amalgamation. 
\end{dfn}

\begin{dfn}
Suppose $H$ is a GHS of $M$. Let $N$ denote a submanifold of $M$ bounded by elements of $\thin{H}$. Then we may define a GHS $H(N)$ of $N$. The thick and thin levels of $H(N)$ are the thick and thin levels of $H$ that lie in $N$. 
\end{dfn}

%\begin{lem}
%Suppose $H$ is a GHS of $M=X \cup_F Y$, where $F \in \thin{H}$. Then the genus of $H$ is equal to the sum of the genera of $H(X)$ and $H(Y)$, minus the genus of $F$. 
%\end{lem}

\begin{lem}
\label{l:GenusSum}
Suppose $H$ is a GHS of $M$, $F$ is an arbitrary subset of $\thin{H}$ in the interior of $M$, and $\{M_i\}_{i=1}^n$ are the closures of the components of $M \setminus F$. Then
\[\gen(H)=\sum \limits _{i=1} ^n \gen (H(M_i)) -\gen(F) +|F|-n+1.\]
\end{lem}

\begin{proof}
The proof is by induction on $|F|$. Suppose first $F$ is connected, so that $|F|=1$. There are then two cases, depending on whether or not $F$ separates $M$. 

We first deal with the case where $F$ separates  $M$ into $M_1$ and $M_2$. In this case $|F|-n+1=1-2+1=0$, so we need to establish \[\gen(H)=\gen(H(M_1))+\gen(H(M_2))-\gen(F).\]

Let $\Sigma (M_1)$ and $\Sigma (M_2)$ denote spines of $\amlg{H(M_1)}$ and $\amlg{H(M_2)}$. Then $\Sigma (M_1)$ is the union of a properly embedded graph $\Sigma (M_1)' \subset M_1$ and $\bdy_- M_1$. If $M_1$ is above $M_2$ in $M$ then $F$ is a component of $\bdy _- M_1$. Let $(\bdy _-M_1)'=\bdy _-M_1 \setminus F$. 

To form the spine of $\amlg{H}$ we attach a vertical arc from $\Sigma (M_2)$ to  $\Sigma (M_1)' \cup (\bdy _- M_1)'$, through the compression body in $M_2$ that is incident to $F$. Hence, the graph part of the spine of $H$ comes from the graph parts of $\Sigma (M_1)$ and $\Sigma(M_2)$, together with an arc. The surface part only comes from the surface part of $\Sigma(M_2)$ and the surface parts of $\Sigma(M_1)$ other than $F$. See Figure \ref{f:AmlgSpine}. Hence, the spine of $\amlg{H}$ is obtained from $\Sigma(M_1) \cup \Sigma(M_2)$ by connecting with a vertical arc and removing a copy of $F$. The result thus follows. 

        \begin{figure}[htbp]
        \psfrag{X}{$M_1$}
        \psfrag{Y}{$M_2$}
        \psfrag{F}{$F$}
        \vspace{0 in}
        \begin{center}
       \includegraphics[width=3 in]{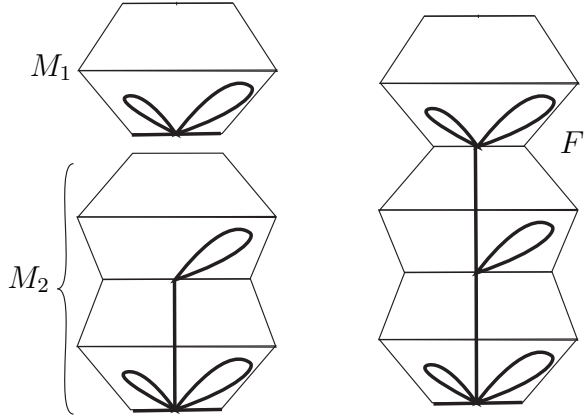}
       \caption{The spine of $\amlg{H}$ is obtained from $\Sigma(M_1) \cup \Sigma(M_2)$ by connecting with a vertical arc and removing a copy of $F$.}
        \label{f:AmlgSpine}
        \end{center}
        \end{figure}

We now move on to the case where $F$ is a connected, non-separating surface. Now $|F|-n+1=1-1+1=1$, so we need to establish \[\gen(H)=\gen(H(M_1))+\gen(H(M_2))-\gen(F)+1.\]
Let $N$ denote the manifold obtained from $M$ by cutting open along $F$. Let $\Sigma(N)$ denote the spine of $\amlg{H(N)}$. As in the separating case, the spine of $\amlg{H}$ is obtained from $\Sigma(N)$ by first removing a copy of $F$. This drops the genus by the genus of $F$. To complete the formation of the spine of $\amlg{H}$, we attach a vertical arc through the compression body incident to the other copy of $F$. As this arc connects what remains of $\Sigma(N)$ to itself, this increases the genus by one. 

To proceed from the case where $|F|=1$ to arbitrary values of $|F|$, simply note that $M$ can be successively built up from $\{M_i\}$ by attaching along one component of $F$ at a time. The result thus follows by an elementary induction argument.
\end{proof}

\begin{cor}
\label{c:GHSgenus}
Let $H$ be a GHS of $M$. Let $\thin{H}^\circ$ denote the subset of $\thin{H}$ consisting of those elements that lie in the interior of $M$. Then
\begin{eqnarray*}
\gen(H)&=&\sum \limits _{H_+ \in \thick{H}} \gen(H_+) -\sum \limits _{H_- \in \thin{H}^\circ} \gen(H_-) \\
&&+ |\thin{H}^\circ| - |\thick{H}|+1
\end{eqnarray*}
\end{cor}

\begin{proof}
Let $F$ be the union of all of the surfaces in $\thin{H}^\circ$, and apply Lemma \ref{l:GenusSum}. Note that there is one element of $\thick{H}$ in each component of the complement of $\thin{H}^\circ$. So the number of such components is precisely $|\thick{H}|$. 
\end{proof}

It should be noted that an alternative approach to the material in this section would be to first define the genus of a GHS to be that given by the formula in Corollary \ref{c:GHSgenus}. Lemma \ref{l:GenusSum} then follows from this definition fairly quickly. However, to prove equivalence to the definition given here, one would need an additional lemma that asserts that genus does not change under weak reductions that are not destabilizations.

\begin{lem}
\label{l:FparallelToThinLevelGHS}
Let $M$ be a 3-manifold which has a $g$-barrier surface $F$. Let $H$ be a genus $g$ strongly irreducible or critical GHS of $M$. Then $F$ is isotopic to a thin level of $H$.
\end{lem}

\begin{proof}
Since the genus of $H$ is $g$, it follows from Corollary \ref{c:GHSgenus} that the genus of every thick and thin level of $H$ is at most $g$. By Theorem \ref{t:IncompressibleThinLevels} we know that each thin level of $H$ is incompressible. Since $F$ is a $g$-barrier surface, it can be isotoped to be disjoint from every thin level. But then $F$ is contained in $M(H_+)$, for some thick level, $H_+$. The surface $H_+$ is either strongly irreducible or critical, so again since $F$ is a $g$-barrier surface it may be isotoped to be disjoint from $H_+$. The surface $F$ can thus be isotoped into a compression body, $\CC$.  But every incompressible surface in $\CC$ is parallel to some component of $\bdy _-\CC$. Each such component is a thin level of $H$. 
\end{proof}

\section{Sequences of GHSs}
\label{s:LastDefSection}

\begin{dfn}
A {\it Sequence Of GHSs} (SOG), $\{H^i\}$ of $M$ is a finite sequence such that for each $i$ either $H^i$ or $H^{i+1}$ is obtained from the other by a weak reduction.
\end{dfn}

\begin{dfn}
If $\bf H$ is a SOG and $k$ is such that $H^{k-1}$ and $H^{k+1}$ are obtained from $H^k$ by a weak reduction then we say the GHS $H^k$ is {\it maximal} in $\bf H$. 
\end{dfn}

It follows that maximal GHSs are larger than their immediate predecessor and immediate successor. 

Just as there are ways to make a GHS ``smaller", there are also ways to make a SOG ``smaller". These are called {\it SOG reductions}, and are explicitly defined in Section 8 of \cite{gordon}. If the first and last GHS of a SOG are strongly irreducible and there are no SOG reductions then the SOG is said to be {\it irreducible}. For our purposes, all we need to know about SOG reduction is that the maximal GHSs of the new SOG are obtained from the maximal GHSs of the old one by weak reduction, and the following lemma holds:

\begin{lem}
\label{l:maximalGHS}
{\rm (\cite{gordon}, Lemma 8.9)} Every maximal GHS of an irreducible SOG is critical. 
\end{lem}

\begin{dfn}
The {\it genus} of a SOG is the maximum among the genera of its GHSs.
\end{dfn}

\begin{lem}
\label{l:GenusGoesDown}
If a SOG $\Lambda$ is obtained from an SOG $\Gamma$ by a reduction then the genus of $\Gamma$ is at least the genus of $\Lambda$. 
\end{lem}

\begin{proof}
Since weak reduction can only decrease the genus of a GHS, the genus of a SOG is the maximum among the genera of its maximal GHSs. But if one SOG is obtained from another by a reduction, then its maximal GHSs are obtained from GHSs of the original by weak reductions. The result thus follows from Lemma \ref{l:AmalgGenus}.
\end{proof}

\begin{lem}
\label{l:FparallelToThinLevelSOG}
Let $M$ be a 3-manifold which has a $g$-barrier surface $F$. Let $\bf H$ be a genus $g$ irreducible SOG of $M$. Then $F$ is isotopic to a thin level of every element of $\bf H$. 
\end{lem}

\begin{proof}
By Lemma \ref{l:maximalGHS} each maximal GHS of $\bf H$ is critical. Hence, by Lemma \ref{l:FparallelToThinLevelGHS} $F$ is isotopic  to a thin level of every maximal GHS of $\bf H$. But every other GHS of $\bf H$ is obtained from a maximal GHS by a sequence of weak reductions and destabilizations. Such moves may create new thin levels, but will never destroy an incompressible thin level. Hence, $F$ is isotopic to a thin level of every element of $\bf H$. 
\end{proof}

\section{Lower bounds on stabilizations.}
\label{s:CounterExamples}

\begin{lem}
\label{l:LowerBoundTheorem}
Let $\{F_i\}_{i=1}^n$ denote a collection of $g$-barrier surfaces in $M$. Let $\{M_k\}_{k=1}^m$ denote the closures of the components of $M-\cup F_i$. Let ${\bf H}=\{H^j\}$ denote an irreducible SOG of $M$. If $F_1$ is isotopic to a unique thin level of $H^1$ and $H^m$, but is oriented in opposite ways in each of these GHSs, then \[\gen({\bf H}) \ge \min \{g, \sum \limits _{k} \gen(M_k)-\sum \limits _{i \ne 1} \gen(F_i)+n-m+1\}.\]
\end{lem}

Here $\gen(M_k)$ denotes the minimal genus among all Heegaard surfaces in $M_k$. 

\begin{proof}
Assume $\gen({\bf H}) \le g$. By Lemma \ref{l:FparallelToThinLevelSOG} the surface $F_1$ is then isotopic to a thin level of every GHS of $\bf H$. Weak reduction can not simultaneously kill one thin level and create a new one, so it follows that for some $j$, there is a GHS $H^j$ of $\bf H$ where $F_1$ is isotopic to two thin levels, but oriented differently. Since these two thin levels are disjoint and isotopic to $F_1$, there must be a manifold $P$ between them homeomorphic to $F_1 \times I$. Let $\overline{P}=\amlg{H^j(P)}$. Then $\overline{P}$ is a Heegaard surface in $P$ that does not separate its boundary components. As $P$ is homeomorphic to $F_1 \times I$, it follows from \cite{st:93} that $\overline{P}$  is a stabilization of two copies of $F_1$, connected by a tube. Hence, $\gen(\overline{P}) \ge 2 \gen(F_1)$. 

By Lemma \ref{l:FparallelToThinLevelSOG}, for each $i$ there is a thin level of $H^j$ which is isotopic to $F_i$. For each $i \ne 1$ choose one such thin level, and call it $F_i^j$. If we cut $M$ along $\{F^j_i|i \ne 1\}$, and then remove the interior of $P$, we obtain a collection of manifolds homeomorphic to $\{M_k\}$. We denote this collection as $\{M^j_k\}$. For each $k$, let $\overline{M_k}=\amlg{H^j(M^j_k)}$. It thus follows from Lemma \ref{l:GenusSum} that

\begin{eqnarray*}
\gen({\bf H}) & \ge & \gen(H^j)\\
&=& \sum \limits _k \gen(\overline{M_k})-\sum \limits _{i \ne 1} \gen(F_i) +\gen(\overline{P}) - 2\gen(F_1)\\
&&\hspace{.5in} +(n+1)-(m+1)+1\\
&\ge &\sum \limits _k \gen(M_k)-\sum \limits _{i \ne 1} \gen(F_i)+n-m+1
\end{eqnarray*}
\end{proof}

In the next three theorems we present our counter-examples to the Stabilization Conjecture. 

\begin{thm}
\label{t:FlipCounterExample}
For each $n \ge 4$ there is a closed, orientable 3-manifold that has a genus $n$ Heegaard surface that must be stabilized at least $n-2$ times to become equivalent to the Heegaard surface obtained from it by reversing its orientation. 
\end{thm}

\begin{proof}
Let $M_1$ and $M_2$ be 3-manifolds (not $I$-bundles) that have one boundary component homeomorphic to a genus $g$ surface, $F$ (where $g \ge 2$). For each $i$ assume the manifold $M_i$ has a strongly irreducible Heegaard surface $H_i$ of genus $g+1$. (It is not difficult to construct examples of such manifolds.)

Now glue $M_1$ and $M_2$ along their boundaries by a ``sufficiently complicated" map, so that by Theorem \ref{t:DistanceBoundTheorem} the gluing surface $F$ becomes a $(2g+2)$-barrier surface. Let $M$ be the resulting 3-manifold. A GHS $H^1$ of $M$ is then defined by:
	\begin{enumerate}
		\item $\thick{H^1}=\{H_1, H_2\}$
		\item $\thin{H^1}=\{F\}$
	\end{enumerate}
	
Choose an orientation on $H^1$. Let $H^*$ denote the GHS with the same thick and thin levels, but with opposite orientation. Then $\amlg{H^*}$ is a Heegaard surface in $M$ that is obtained from the Heegaard surface  $\amlg{H^1}$ by reversing its orientation. By Corollary \ref{c:GHSgenus} the genera of these surfaces is \[n=2(g+1)-g=g+2.\] We now claim that these Heegaard surfaces are not equivalent after any less than $g=n-2$ stabilizations. Let $H$ denote the minimal genus common stabilization of these Heegaard surfaces. We must show $\gen(H) \ge (g+2)+g=2g+2$. 

Let ${\bf H}=\{H^i\}_{i=1}^n$ be the SOG where 
	\begin{enumerate}
		\item $H^1$ is as defined above, 
		\item $H^n=H^*$,
		\item for some $1<j<n$, $\thick{H^j}=\{H\}$ and $\thin{H^j}=\emptyset$, and
		\item $H^j$ is the only maximal GHS in $\bf H$. 
	\end{enumerate}

Let ${\bf K}$ be a SOG obtained from $\bf H$ by a maximal sequence of SOG reductions. By Lemma \ref{l:GenusGoesDown}, $\gen({\bf H}) \ge \gen({\bf K})$. When the initial and final GHS of a SOG are strongly irreducible, SOG reduction will leave them unaffected. Hence, since the orientations on $F$ disagree in the initial and final GHS of $\bf H$, this must also be true of $\bf K$. By Lemma \ref{l:LowerBoundTheorem} we then have
\begin{eqnarray*}
\gen({\bf K}) & \ge & \gen(M_1) + \gen(M_2)\\
&=&(g+1)+(g+1)\\
&=&2g+2
\end{eqnarray*}
Hence, $\gen(H) = \gen({\bf H}) \ge \gen({\bf K}) \ge 2g+2$. 
\end{proof}

\begin{thm}
\label{t:TorusBoundaryCounterExamples}
For each $n \ge 5$ there is an orientable 3-manifold whose boundary is a torus, that has two genus $n$ Heegaard surfaces which must be stabilized at least $n-4$ times to become equivalent.
\end{thm}

\begin{proof}
Let $M_1$ and $M_2$ be 3-manifolds (that are not $I$-bundles) that have one boundary component homeomorphic to a genus $g$ surface, $F$ (where $g \ge 2$). The manifold $M_1$ also has a boundary component $T$ that is a torus. The manifold $M_2$ has no boundary components other than $F$. For each $i$, assume the manifold $M_i$ has a strongly irreducible Heegaard surface $H_i$ of genus $g+1$. The manifold $M_1$ then has a genus $g+2$ Heegaard surface $G_1$, obtained from $H_1$ by boundary stabilizing along $T$. 

Now glue $M_1$ and $M_2$ along their genus $g$ boundary components by a ``sufficiently complicated" map, so that by Theorem \ref{t:DistanceBoundTheorem} the gluing surface $F$ becomes a $(2g+2)$-barrier surface. Let $M$ be the resulting 3-manifold. GHSs $H^1$ and $H^*$  of $M$ are then defined by:
	\begin{enumerate}
		\item $\thick{H^1}=\{H_1, H_2\}$
		\item $\thick{H^*}=\{G_1,H_2\}$
		\item $\thin{H^1}=\thin{H^*}=\{F, T\}$
	\end{enumerate}

Choose orientations on $H^1$ and $H^*$ so that the orientations on $T$ agree. Then both $\amlg{H^1}$ and $\amlg{H^*}$ are Heegaard surfaces in $M$, with $T$ on the same side of each.  Hence, these two Heegaard surfaces have some common stabilization, $H$. Note also that the orientations on $F$ in $H^1$ and $H^*$ necessarily disagree. See Figure \ref{f:H1H*Orientations}. Let $H^{**}$ denote the GHS obtained from $H^*$ by a maximal sequence of weak reductions. As weak reduction cannot kill an incompressible thin level, $F$ is a thin level of $H^{**}$ that is oriented oppositely in $H^{**}$ than in $H^1$.

        \begin{figure}[htbp]
        \psfrag{1}{$H^1$}
        \psfrag{2}{$H^*$}
        \psfrag{F}{$F$}
        \psfrag{H}{$H_1$}
        \psfrag{G}{$G_1$}
        \psfrag{h}{$H_2$}
        \psfrag{T}{$T$}
        \vspace{0 in}
        \begin{center}
       \includegraphics[width=4.5 in]{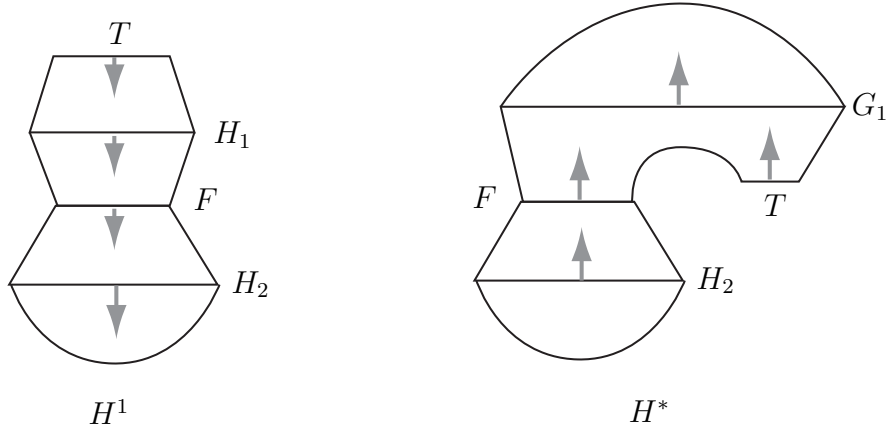}
       \caption{If the orientations on $T$ in $H^1$ and $H^*$ agree, then the orientations on $F$ disagree.}
        \label{f:H1H*Orientations}
        \end{center}
        \end{figure}

By Corollary \ref{c:GHSgenus} the genus of $\amlg{H^1}$  is \[2(g+1)-g=g+2.\] The genus of $\amlg{H^*}$ is one higher, $g+3$. Let this number be $n$. We now claim that we must stabilize $\amlg{H^*}$ at least $g-1=n-4$ times to obtain a stabilization of $\amlg{H^1}$. In other words, we claim \[\gen(H) \ge (g+3)+(g-1)=2g+2.\] 

Let ${\bf H}=\{H^i\}_{i=1}^m$ be the SOG where 
	\begin{enumerate}
		\item $H^1$ is as defined above, 
		\item $H^n=H^*$,
		\item $H^{n+m}=H^{**}$,
		\item for some $1<j<n$, $\thick{H^j}=\{H\}$ and $\thin{H^j}=T$, and
		\item $H^j$ is the only maximal GHS in $\bf H$. 
	\end{enumerate}

Let ${\bf K}$ be a SOG obtained from $\bf H$ by a maximal sequence of SOG reductions. By Lemma \ref{l:GenusGoesDown}, $\gen({\bf H}) \ge \gen({\bf K})$. When the initial and final GHS of a SOG are strongly irreducible, SOG reduction will leave them unaffected.  Since the orientations on $F$ disagree in the initial and final GHS of $\bf H$, this must also be true of $\bf K$. Hence, by Lemma \ref{l:LowerBoundTheorem}, 
\begin{eqnarray*}
\gen({\bf K}) & \ge & \gen(M_1) + \gen(M_2)\\
&=&(g+1)+(g+1)\\
&=&2g+2
\end{eqnarray*}
Hence, $\gen(H) = \gen({\bf H}) \ge \gen({\bf K}) \ge 2g+2$. 
\end{proof}

\begin{thm}
\label{t:ClosedCounterExamples}
For each $n \ge 8$ there is a closed, orientable 3-manifold that has a pair of genus $n$ Heegaard surfaces which must be stabilized at least $\frac{1}{2}n -3$ times to become equivalent (regardless of their orientations).
\end{thm}

\begin{proof}
Let $M_1$, $M_2$, $M_3$, and $M_4$ be 3-manifolds, none of which is an $I$-bundle, as follows. Each of these manifolds has one boundary component homeomorphic to a genus $g$ surface, $F$ (where $g \ge 2$), and a Heegaard surface $H_i$ of genus $g+1$ that separates $F$ from any other boundary component. The manifolds $M_1$ and $M_2$ have a second boundary component, which is a torus. The manifold $M_3$ has two toroidal boundary components. The manifold $M_4$ has no boundary components other than $F$. For $i=1$ and $2$ the manifolds $M_i$  also have a second Heegaard surface, $G_i$, of genus $g+2$ obtained from $H_i$ by boundary stabilizing along the torus boundary component. 

Now glue all four manifolds together as in Figure \ref{f:H1GHS} by ``sufficiently complicated" maps so that by Theorem \ref{t:DistanceBoundTheorem} both copies of $F$, and both gluing tori, become $(3g+3)$-barrier surfaces. Let $M$ be the resulting 3-manifold. For $i=1$ and 2 let $T_i$ denote the torus between $M_i$ and $M_3$. Let $F_1$ denote the copy of $F$ between $M_1$ and $M_2$, and $F_2$ the copy of $F$ between $M_3$ and $M_4$. 

        \begin{figure}[htbp]
        \psfrag{a}{$M_1$}
        \psfrag{b}{$M_2$}
        \psfrag{c}{$M_3$}
        \psfrag{d}{$M_4$}
        \psfrag{G}{$G_1$}
        \psfrag{g}{$G_2$}
        \psfrag{t}{$T_1$}
        \psfrag{t}{$T_1$}
        \psfrag{T}{$T_2$}
        \psfrag{f}{$F_1$}
        \psfrag{F}{$F_2$}
        \psfrag{1}{$H_1$}
        \psfrag{2}{$H_2$}
        \psfrag{3}{$H_3$}
        \psfrag{4}{$H_4$}
        \vspace{0 in}
        \begin{center}
       \includegraphics[width=4.5 in]{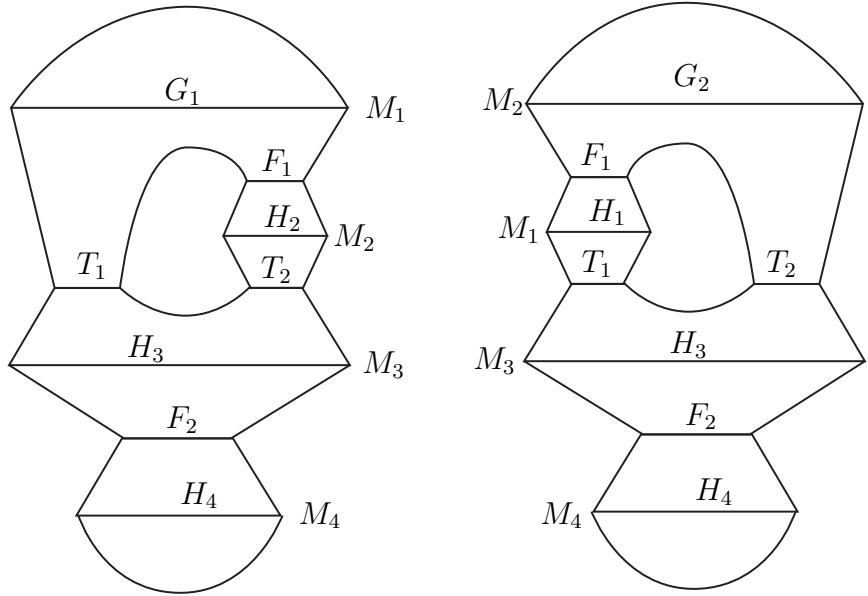}
       \caption{The GHSs, $H^1$ and $H^*$.}
        \label{f:H1GHS}
        \end{center}
        \end{figure}

We now define two GHSs $H^1$ and $H^*$ of $M$ (See Figure \ref{f:H1GHS}):

\begin{enumerate}
	\item $\thick{H^1}=\{G_1, H_2, H_3, H_4\}$
	\item $\thick{H^*}=\{H_1, G_2, H_3, H_4\}$. 
	\item $\thin{H^1}=\thin{H^*}=\{F_1, F_2, T_1, T_2\}$
\end{enumerate}

By definition, $\amlg{H^1}$ and $\amlg{H^*}$ are both Heegaard surfaces in $M$. By Corollary \ref{c:GHSgenus} the genera of these surfaces is \[n=3(g+1)+(g+2)-2g-2+1=2g+4.\] We claim that no matter what orientation is chosen for these GHSs, they are not equivalent after any less than $g-1=\frac{1}{2}n-3$ stabilizations. Let $H$ denote the minimal genus common stabilization of these Heegaard surfaces. We must show $\gen(H) \ge (2g+4)+(g-1)=3g+3$. 

Orient $H^1$ and $H^*$. Note that if these orientations agree on $F_1$ then they disagree on $F_2$. See Figure \ref{f:F1F2Orientations}. Hence, any SOG that interpolates between $H^1$ and $H^*$ must reverse the orientation of either $F_1$ or $F_2$.

        \begin{figure}[htbp]
        \psfrag{a}{(a)}
        \psfrag{b}{(b)}
        \psfrag{1}{$H^1$}
        \psfrag{f}{$F_1$}
        \psfrag{F}{$F_2$}
        \vspace{0 in}
        \begin{center}
       \includegraphics[width=4.5 in]{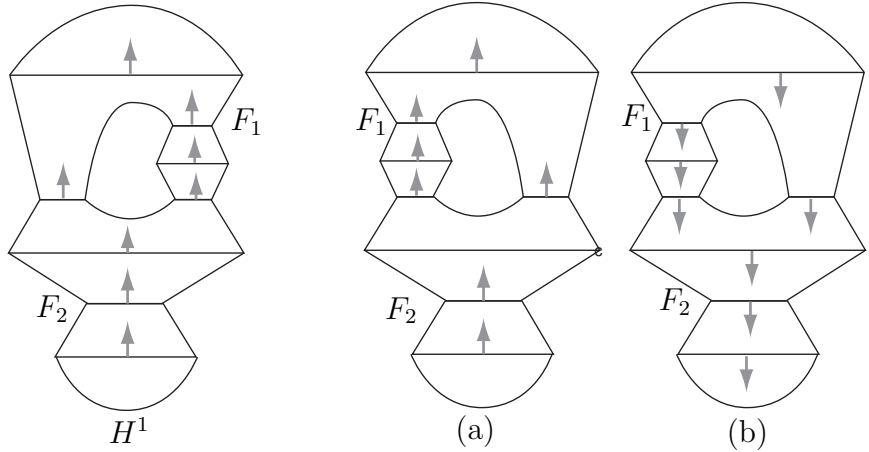}
       \caption{An orientation on $H^1$ and two possible orientations on $H^*$. In $H^1$ the manifold $M_1$ is above $F_1$. In Case (a) the manifold $M_1$ is below $F_1$. Hence, the orientations on $F_1$ in $H^1$ and $H^*$ disagree. In Case (b) the orientations on $F_2$ disagree.}
        \label{f:F1F2Orientations}
        \end{center}
        \end{figure}

Let ${\bf H}=\{H^i\}_{i=1}^n$ be the SOG where 
	\begin{enumerate}
		\item $H^1$ is as defined above, 
		\item $H^n=H^*$,
		\item for some $1<j<n$, $\thick{H^j}=\{H\}$ and $\thin{H^j}=\emptyset$, and
		\item $H^j$ is the only maximal GHS in $\bf H$. 
	\end{enumerate}

Let ${\bf K}=\{K^i\}$ be a SOG obtained from $\{H^i\}$ by a maximal sequence of SOG reductions. By Lemma \ref{l:GenusGoesDown}, $\gen({\bf H}) \ge \gen({\bf K})$. By Lemma \ref{l:LowerBoundTheorem}, 
\begin{eqnarray*}
\gen({\bf K}) & \ge & \sum \limits_{i=1} ^4 \gen(M_i) - \gen(T_1)-\gen(T_2)-\gen(F)+1\\
&=&4(g+1)-2-g+1\\
&=&3g+3
\end{eqnarray*}
Hence, $\gen(H) = \gen({\bf H}) \ge \gen({\bf K}) \ge 3g+3$. 
\end{proof}

\section{Amalgamations of unstabilized Heegaard splittings}
\label{s:Amalgamations}

\begin{thm}
\label{t:HigherGenusGordon}
Let $M_1$ and $M_2$ be compact, orientable, irreducible 3-manifolds with incompressible boundary, neither of which is an $I$-bundle. Let $M$ denote the manifold obtained by gluing some component $F$ of $\bdy M_1$ to some component of $\bdy M_2$ by some homeomorphism $\phi$. Let $H_i$ be an unstabilized, boundary-unstabilized Heegaard surface in $M_i$. If $\phi$ is sufficiently complicated then the amalgamation of $H_1$ and $H_2$ in $M$ is unstabilized.
\end{thm}

Here the term ``sufficiently complicated" means that the distance of $\phi$ is high enough so that by Theorem \ref{t:DistanceBoundTheorem} the surface $F$ becomes a $g$-barrier surface, where $g=\mbox{genus}(H_1)+\mbox{genus}(H_2)-\mbox{genus}(F)$.

\begin{proof}
Let ${\bf \Gamma}$ be the SOG depicted in Figure \ref{f:InitialSOGgordon}. The second GHS pictured is the one whose thick levels are $H_1$ and $H_2$. The first GHS in the figure is obtained from this one by a maximal sequence of weak reductions. The third GHS is the one whose only thick level is the amalgamation $H$ of $H_1$ and $H_2$. The next GHS pictured is obtained from $H$ by some number of destabilizations. Finally, the last GHS is obtained from the second to last by a maximal sequence of weak reductions. Note that by construction, $\mbox{genus}({\bf \Gamma})=\mbox{genus}(H)=g$. (The second equality follows from Corollary \ref{c:GHSgenus}.)

        \begin{figure}[htbp]
        \psfrag{X}{$H_2$}
        \psfrag{H}{$H_1$}
        \psfrag{F}{$F$}
        \psfrag{1}{$H$}
        \psfrag{2}{$G$}
        \vspace{0 in}
        \begin{center}
       \includegraphics[width=3.5 in]{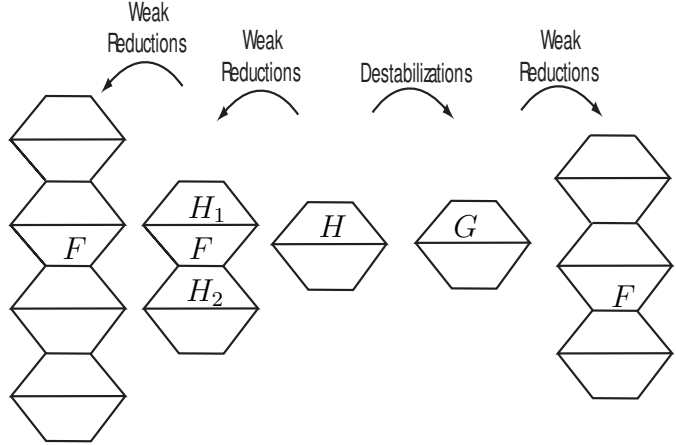}
       \caption{The initial SOG, $\bf \Gamma$.}
        \label{f:InitialSOGgordon}
        \end{center}
        \end{figure}

Now let ${\bf \Lambda}=\{\Lambda ^i\}_{i=1}^n$ be the SOG obtained from ${\bf \Gamma}$ by a maximal sequence of SOG reductions. When the first and last GHS of a SOG admit no weak reductions, then they remain unaffected by SOG reduction. Hence, $\Lambda^1$ is the first element of ${\bf \Gamma}$ and $\Lambda ^n$ is the last element of ${\bf \Gamma}$.

Since $F$ is a $g$-barrier surface, it is isotopic to a thin level of every GHS of $\bf \Lambda$. Let $m$ denote the largest number such that $F$ is isotopic to a {\it unique} thin level $F_i$ of $\Lambda ^i$, for all $i \le m$. The surface $F_i$ then divides $M$ into manifolds $M_1^i$ and $M_2^i$, homeomorphic to $M_1 $ and $M_2$, for each $i \le m$. 

Now note that there are no stabilizations in the original SOG $\bf \Gamma$. It thus follows from Lemma 8.12 of \cite{gordon} that the first destabilization in $\bf \Lambda$ happens before the first stabilization. Furthermore, as the genus of $\Lambda ^n$ is less than the genus of $\Lambda ^1$, there is at least one destabilization in $\bf \Lambda$. Let $p$ denote the smallest value for which $\Lambda^{p+1}$ is obtained from $\Lambda ^p$ by a destabilization. Then for all $i\le p$, either $\Lambda^i$ or $\Lambda ^{i-1}$ is obtained from the other by a weak reduction that is not a destabilization. 

Suppose first $p \le m$. Then for all $i \le p$ the manifold $M_1^i$ is defined (because $i \le m$), and either $\Lambda^i(M_1^i)=\Lambda^{i-1}(M_1^{i-1})$ or one of $\Lambda^i(M_1^i)$ and $\Lambda^{i-1}(M_1^{i-1})$ is obtained from the other by a weak reduction that is not a destabilization (because $i \le p$). 
%Say why this is true!! Non-trivial: I am saying something here about lack of boundary stabilizations/destabilizations in $M_1$.
It follows from Lemma \ref{l:AmalgGenus} that $H_1^i=\amlg{\Lambda^i(M_1^i)}$ is the same for all $i\le p$. But $H_1^1=H_1$, so $H_1^p=H_1$. By identical reasoning $H_2^p=\amlg{\Lambda^p(M_2^p)}=H_2$. Since $H_1$ and $H_2$ are unstabilized, $H_1^p$ and $H_2^p$ must be unstabilized as well. Hence, neither $H_1^{p+1}$ nor $H_2^{p+1}$  can be obtained from $H_1^p$ or $H_2^p$ by destabilization, a contradiction. 

We thus conclude $p > m$, and thus $H_1^m=H_1$ and $H_2^m=H_2$. In particular, it follows that $m$ is strictly less than $n$. That is, there exists a GHS $\Lambda ^{m+1}$ which has two thin levels isotopic to $F$. 

Since $\Lambda ^{m+1}$ has a thin level that is not a thin level of $\Lambda^m$, it must be obtained from $\Lambda^m$ by a weak reduction. It follows that there is some thin level $F_{m+1}$ of $\Lambda ^{m+1}$ that is identical to $F_m$. The other thin level of $\Lambda ^{m+1}$ that is isotopic to $F$ we call $F_{m+1}'$. The surface $F_{m+1}'$ either lies in $M_1^m$ or $M_2^m$. Assume the former. Let $M^{m+1}_1$ denote the side of $M$ cut along $F_{m+1}$ homeomorphic to $M_1$. It follows that $\Lambda^{m+1}(M_1^{m+1})$ is obtained from $\Lambda^m(M^m_1)$ by a weak reduction that is not a destabilization. Thus, by  Lemma \ref{l:AmalgGenus}, \[H^{m+1}_1=\amlg{\Lambda^{m+1}(M^{m+1}_1)}=  \amlg{\Lambda^m(M^m_1)}=H_1^m=H_1.\] 

The surfaces $F_{m+1}$ and $F_{m+1}'$ cobound a product region $P$ of $M$. A GHS of $P$ is given by $\Lambda^{m+1}(P)$, and thus $H_P=\amlg{\Lambda^{m+1}(P)}$ is a Heegaard surface in a product. If this Heegaard surface is stabilized, then $H^{m+1}_1$ would be stabilized. But since $H^{m+1}_1=H_1$, and $H_1$ is unstabilized, this is not the case. 

We conclude $H_P$ is an unstabilized Heegaard surface in $P$. By \cite{st:93} such a Heegaard surface admits no weak reductions, and thus  $H_P$ must be the unique thick level of $\Lambda^{m+1}(P)$. From \cite{st:93} this Heegaard surface is either  a copy of $F$, or two copies of $F$ connected by a single unknotted tube. In the former case we have a contradiction, as the thick level of  $\Lambda^{m+1}(P)$ would be parallel to the two thin levels $F_{m+1}$ and $F_{m+1}'$, and would thus have been removed during weak reduction. In the latter case $H^{m+1}_1$ is boundary-stabilized. As this Heegard surface is $H_1$, which is not boundary-stabilzed, we again have a contradiction. 
\end{proof}

An example of a 3-manifold that has a weakly reducible, yet unstabilized Heegaard surface which is not a minimal genus Heegaard surface has been elusive. In the next corollary we use Theorem \ref{t:HigherGenusGordon} to construct manifolds that have arbitrarily many such surfaces. 

\begin{cor}
\label{c:Moriah}
There exist manifolds that contain arbitrarily many non-minimal genus, unstabilized Heegaard surfaces which are not strongly irreducible.
\end{cor}

\begin{proof}
Let $M$ denote a 3-manifold with torus boundary, and strongly irreducible Heegaard surfaces of arbitrarily high genus. (Such an example has been constructed by Casson and Gordon. See \cite{sedgwick:97}. The manifold they construct is closed, but there is a solid torus that is a core of one of the handlebodies bounded by each Heegaard surface. Thus, removing this solid torus produces  a manifold with torus boundary that has arbitrarily high genus strongly irreducible Heegaard surfaces.) 

Now let $M_1$ and $M_2$ be two copies of $M$, and let $H_g^i$ denote a genus $g$ strongly irreducible surface in $M_i$. As $H_g^i$ is strongly irreducible, it is neither stabilized nor boundary-stabilized. Hence, if $M_1$ is glued to $M_2$ by a sufficiently complicated homeomorphism, it follows from Theorem \ref{t:HigherGenusGordon} that the amalgamation of $H_g^1$ and $H_g^2$ is unstabilized, for all $g \le G$. (One can make the genus of $G$ as high as desired without changing the genus of $M_1 \cup M_2$ by gluing $M_1$ to $M_2$ by more and more complicated maps.)

Finally, note that every amalgamation is weakly reducible. 
\end{proof}

\section{Low genus Heegaard surfaces are amalgamations}

In this section we establish a refinement of a result due independently to Lackenby \cite{lackenby:04}, Souto \cite{souto}, and Li \cite{li:08}. Their result says that if 3-manifolds $M_1$ and $M_2$ are glued by a sufficiently complicated map, then all low genus Heegaard surfaces in the resulting manifold are amalgamations of Heegaard surfaces in $M_1$ and $M_2$. 

\begin{thm}
\label{t:AmalgamationExists}
Let $M_1$ and $M_2$ be compact, orientable, irreducible 3-manifolds with incompressible boundary, neither of which is an $I$-bundle. Let $M$ denote the manifold obtained by gluing some component $F$ of $\bdy M_1$ to some component of $\bdy M_2$ by some homeomorphism $\phi$. If $\phi$ is sufficiently complicated then any low genus, unstabilized Heegaard surface $H$ in $M$ is an amalgamation of unstabilized, boundary-unstabilized surfaces in $M_1$ and $M_2$, and possibly a Type II splitting of $F \times I$.
\end{thm}

Here the terms ``sufficiently complicated" and ``low genus" mean that the distance of $\phi$ is high enough so that by Theorem \ref{t:DistanceBoundTheorem} the surface $F$ becomes a $g$-barrier surface, where $g=\mbox{genus}(H)$.

\begin{proof}
Let $H_*$ be an unstabilized Heegaard surface in $M$ whose genus is at most $g$. Let $H$ be a GHS obtained from the GHS whose only thick level is $H_*$ by a maximal sequence of weak reductions (Figure \ref{f:AmalgamExists}(b)).  Since $H_*$ was unstabilized, it follows from Lemma \ref{l:AmalgGenus} that $\amlg{H}=H_*$. 

        \begin{figure}[htbp]
        \psfrag{H}{$H_*$}
        \psfrag{F}{$F$}
        \psfrag{a}{$G_1$}
        \psfrag{b}{$G_2$}
        \psfrag{g}{$H_1$}
        \psfrag{G}{$H_2$}
        \psfrag{h}{$H_F$}
        \psfrag{1}{(a)}
        \psfrag{2}{(b)}
        \psfrag{3}{(c)}
        \psfrag{4}{(d)}
        \psfrag{5}{(e)}
        \vspace{0 in}
        \begin{center}
       \includegraphics[width=3.5 in]{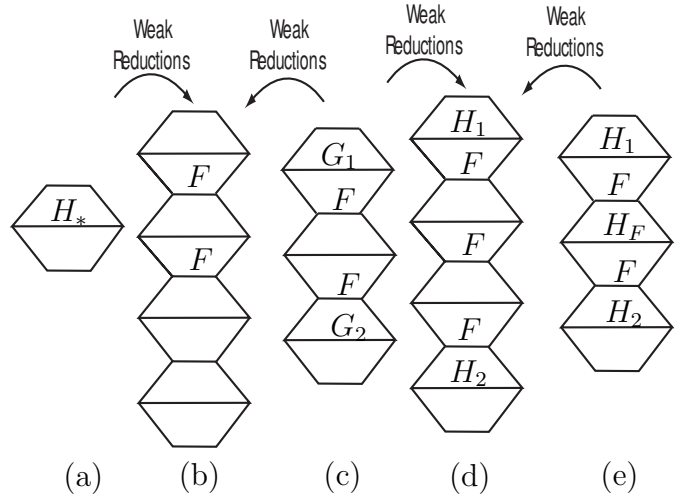}
       \caption{The GHSs of the proof of Theorem \ref{t:AmalgamationExists}.}
        \label{f:AmalgamExists}
        \end{center}
        \end{figure}

By Theorem \ref{t:DistanceBoundTheorem}, $F$ is a $g$-barrier surface. Hence, $F$ is isotopic to at least one thin level of $H$. Now cut $M$ along all thin levels isotopic to $F$. The result is manifolds $M_1'$ and $M_2'$ homeomorphic to $M_1$ and $M_2$, and possibly several manifolds homeomorphic to $F \times I$. The Heegaard surface $H_*=\amlg{H}$ is thus an amalgamation of the surfaces $G_1=\amlg{H(M_1')}$ and $G_2=\amlg{H(M_2')}$, and possibly a Heegaard surface in $F \times I$ (Figure \ref{f:AmalgamExists}(c)).

Since $H_*$ is unstabilized, it follows that both $G_1$ and $G_2$ are unstabilized. Now suppose that $G_i$ is boundary-stabilized. Then $G_i$ is the amalgamation of an unstabilized, boundary-unstabilized Heegaard surface $H_i$ in $M_i'$, and a Heegard surface in $F \times I$. If $G_i$ was boundary-unstabilized to begin with, then let $H_i=G_i$. Thus, $H_*$ is an amalgamation of $H_1$, $H_2$, and possibly multiple Heegaard surfaces in $F \times I$ (Figure \ref{f:AmalgamExists}(d)), which can again be amalgamated to a single Heegaard surface $H_F$ in $F \times I$ (Figure \ref{f:AmalgamExists}(e)). 

By \cite{st:93} $H_F$ is a stabilization of either a copy of $F$ (i.e.~a stabilization of a Type I splitting), or of two copies of $F$ connected by a vertical tube (i.e.~a stabilization of a Type II splitting). However, our assumption that $H_*$ was unstabilized implies $H_F$ is unstabilized. Furthermore, as $H_F$ comes from amalgamating non-trivial splittings of $F \times I$,  it will not be a Type I splitting. We conclude that 
the only possibility is that $H_F$ is a Type II splitting of $F \times I$. 
\end{proof}

\section{Isotopic Heegaard surfaces in amalgamated 3-manifolds.}
\label{s:Isotopy}

In Theorem \ref{t:AmalgamationExists} we showed that when $\phi$ is sufficiently complicated then any low genus, unstabilized Heegaard surface $H$ in $M_1 \cup _\phi M_2$ is an amalgamation of unstabilized, boundary unstabilized surfaces $H_1$ and $H_2$ in $M_1$ and $M_2$, and possibly a Type II splitting of $\bdy M_1 \times I$. In the next theorem we show that when there is no Type II splitting in this decomposition, then $H_1$ and $H_2$ are completely determined by $H$.

\begin{thm}
\label{t:HighGenusGordonIsotopy}
Let $M_1$ and $M_2$ be compact, orientable, irreducible 3-manifolds with incompressible boundary, neither of which is an $I$-bundle. Let $M$ denote the manifold obtained by gluing some component $F$ of $\bdy M_1$ to some component of $\bdy M_2$ by some homeomorphism $\phi$. Suppose $\phi$ is sufficiently complicated, and some low genus Heegaard surface  $H$ in $M$ can be expressed as an amalgamation of unstabilized, boundary-unstabilized surfaces in $M_1$ and $M_2$. Then this expression is unique.
\end{thm}

As in Theorem \ref{t:AmalgamationExists}, the terms ``sufficiently complicated" and ``low genus" mean that the distance of $\phi$ is high enough so that by Theorem \ref{t:DistanceBoundTheorem} the surface $F$ becomes a $g$-barrier surface, where $g=\mbox{genus}(H)$.

\begin{proof}
Suppose $H$ can be expressed as an amalgamation of unstabilized, boundary-unstabilized surfaces $H_1$ and $H_2$ in $M_1$ and $M_2$. Suppose also $H$ can be expressed as an amalgamation of unstabilized, boundary-unstabilized surfaces $G_1$ and $G_2$ in $M_1$ and $M_2$. 

Let ${\bf \Gamma}$ be the SOG depicted in Figure \ref{f:InitialSOGisotopy}. The third GHS in the figure is the one whose only thick level is $H$. The second GHS pictured is the GHS whose thick levels are $H_1$ and $H_2$. The first GHS in the figure is obtained from this one by a maximal sequence of weak reductions. The fourth GHS is the one whose thick levels are $G_1$ and $G_2$. Finally, the last GHS is obtained from the fourth by a maximal sequence weak reductions.

        \begin{figure}[htbp]
        \psfrag{X}{$H_2$}
        \psfrag{H}{$H_1$}
        \psfrag{x}{$G_2$}
        \psfrag{h}{$G_1$}
        \psfrag{G}{$H$}
        \psfrag{F}{$F$}
        \vspace{0 in}
        \begin{center}
       \includegraphics[width=3.5 in]{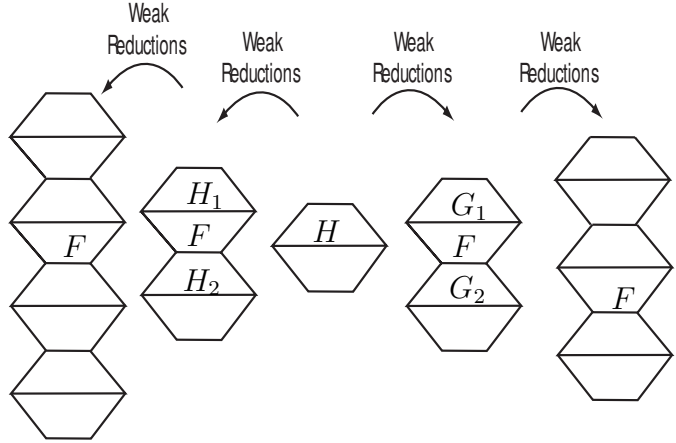}
       \caption{The initial SOG, $\bf \Gamma$.}
        \label{f:InitialSOGisotopy}
        \end{center}
        \end{figure}

Now let ${\bf \Lambda}=\{\Lambda ^i\}_{i=1}^n$ be the SOG obtained from ${\bf \Gamma}$ by a maximal sequence of SOG reductions. When the first and last GHS of a SOG admit no weak reductions, then they remain unaffected by SOG reduction. Hence, $\Lambda^1$ is the first element of ${\bf \Gamma}$ and $\Lambda ^n$ is the last element of ${\bf \Gamma}$.

Note that every GHS of $\bf \Gamma$ is obtained from $H$ by a sequence of weak reductions. By Theorem \ref{t:HigherGenusGordon} the Heegaard surface $H$ is unstabilized, and thus every GHS of $\bf \Gamma$ is unstabilized. Furthermore, every GHS of $\bf \Lambda$ is obtained from GHSs of $\bf \Gamma$ by weak reductions. Hence, every GHS of $\bf \Lambda$ is unstabilized. It follows that there are no destabilizations in $\bf \Lambda$. 

Since $F$ is a $g$-barrier surface, it is isotopic to a thin level of every GHS of $\bf \Lambda$. If, for some $i$, we assume the surface $F$ is isotopic to two elements of $\thin{\Lambda^i}$, then the argument given in the proof of Theorem \ref{t:HigherGenusGordon} provides a contradiction. (This is where we use the assumption that $H_1$ and $H_2$ are not boundary-stabilized.)

We conclude, then, that for each $i$ either $\Lambda^i$ or $\Lambda ^{i+1}$ is obtained from the other by a weak reduction that is not a destabilization. Furthermore, since for all $i$ the surface $F$ is isotopic to a unique thin level of $\Lambda ^i$, it follows that for each $i$,  $M_1(\Lambda^i)=M_1(\Lambda ^{i+1})$, or either $M_1(\Lambda^i)$ or $M_1(\Lambda ^{i+1})$ is obtained from the other by a weak reduction that is not a destabilization. It thus follows from Lemma \ref{l:AmalgGenus} that for each $i$ the surface $\amlg{M_1(\Lambda ^i)}$ is the same (up to isotopy). But $\amlg{M_1(\Lambda ^1)}=H_1$ and $\amlg{M_1(\Lambda ^n)}=G_1$. Hence, $H_1$ is isotopic to $G_1$. A symmetric argument shows $H_2$ must be isotopic to $G_2$, completing the proof. 
\end{proof}

\bibliographystyle{alpha}

\end{document}